 \documentclass[final,3p,times]{elsarticle}

\usepackage{enumerate}
\usepackage{amsmath}
\usepackage{amssymb}
\usepackage{amsthm}
\usepackage{amsfonts}
\usepackage{microtype}
\usepackage{mathrsfs}
\usepackage{graphicx}
\usepackage{color}
\usepackage{latexsym}
\usepackage{rotating}
\usepackage{xspace}
\usepackage[all]{xy}
\usepackage{longtable}
\usepackage{leftidx}
\usepackage{mathtools}
\usepackage{tikz}
\usepackage{geometry}

\usepackage{calc}
\usepackage{tikz}
\usetikzlibrary{decorations.markings}
\tikzstyle{vertex}=[circle, draw, inner sep=0pt, minimum size=6pt]

\newtheorem{thm}{Theorem}[section]

\newtheorem{lem}[thm]{Lemma}

\newtheorem{dfn}[thm]{Definition}
\newtheorem{pro}[thm]{Proposition}
\newtheorem{rem}[thm]{Remark}

=22cm\headheight=0cm\topskip=0cm\topmargin=0cm
\def\a{\bar {a}}
\def\x{\bar {x}}

\def\e{\mathbf {e}}
\def\Ac{\mathcal {A}}\def\Bc{\mathcal {B}}
\def\Cc{\mathcal {C}}

\def\Rn{\mathbb R}

\begin{document}

\begin{frontmatter}

\title{Uniform logical proofs for Riesz representation theorem, Daniell-Stone theorem and Stone's representation theorem for probability algebras}

\author[aut,ipm]{Alireza Mofidi}
\ead{mofidi@aut.ac.ir}

\address[aut]{Department of Mathematics and Computer Science, Amirkabir University of Technology (Tehran Polytechnic), \\ Hafez Avenue, 15194 Tehran, Iran}

\address[ipm]{School of Mathematics, Institute for Research in Fundamental Sciences (IPM),
\\ P.O. Box: 19395-5746, Tehran, Iran}

\begin{abstract}

Riesz representation theorem, Daniell-Stone theorem for Daniell integrals and Stone's representation theorem for probability and measure algebras are three important classical results in analysis concerning existence of measures with certain properties.
Many proofs of these theorems can be found in the literature of analysis, from elementary ones which use ordinary techniques from measure theory, to more sophisticated ones, such as those employing techniques from nonstandard analysis, in particular for Riesz representation theorem.
In this paper, as the first goal, we give new proofs for all these three theorems.
Our proofs have a mild logical flavor and are uniform in the sense that they are all based on the same general idea and rely on the application of the same technical tool from logic to measure theory, namely logical compactness theorem. In fact, as the second goal of the paper, we try to reveal more the power of logical methods in analysis in particular measure theory, and make stronger connections between analysis and logic.
We use the setting of "integration logic" which is a logical framework (and one of the forms of probability logics) for studying measure and probability structures by logical means.
Indeed, we elaborate this setting and use its expressive power and a version of compactness theorem holding in it to show its application in measure theory by giving new proofs for the above-mentioned measure existence theorems.
As mentioned, an advantage of these proofs is that they are all given in a uniform way since they are all based on the logical compactness theorem.
The paper is mostly written for
general mathematicians, in particular
the people active in analysis or logic as the main audience.
So it is self-contained and the reader does not need to have any advanced prerequisite knowledge from logic or measure theory.

\end{abstract}

\begin{keyword}

Riesz representation theorem \sep
Daniell-Stone theorem for Daniell integrals \sep Stone's representation theorem for probability algebras \sep logical compactness theorem \sep integration logic \sep measure existence theorems.

{\sc MSC codes: 28C05 \sep 28A60 \sep 03C98 \sep 03C65.}

\end{keyword}

\end{frontmatter}

\section{Introduction}\label{sectionintroduction}

There are several results in analysis concerning existence of measures with certain properties.
Riesz representation theorem, Daniell-Stone theorem for Daniell integrals and Stone's representation theorem for probability algebras
are some remarkable examples of such measure existence results.
Many proofs for these classical theorems, in particular for Riesz representation theorem, have been discovered by different methods so far.
It is worth mentioning that among various proofs of Riesz representation theorem, many of them, such as the ones in the papers \cite{GarlingAshortproofRiesz}, \cite{GarlingAnothershortproofRiesz}, \cite{HartigTheRieszrepresentationrevisited}, \cite{RossYetanotherproof} and \cite{ZivaljevicALoebmeasureapproachtoRiesz} are using ideas from
outside of classical analysis for example from nonstandard analysis.
On the other hand, there are many instances that tools from mathematical logic are used for studying objects or theories from analysis, probability theory, dynamical systems, etc.
For example in papers such as \cite{ModeltheoryofoperatoralgebrasIFarahHartSherman}, operator algebras have been studied by logical (model theoretic) means.
Also in \cite{FajardoModeltheoryofstochasticprocesses}, an extensive logical investigation of stochastic processes has been carried out.
Furthermore, in \cite{GoldbringTowsnerAnapproximatelogicformeasures} logic gets involved with dynamical systems for giving a proof of the Furstenberg correspondence between finite sets and dynamical systems.
During the course of our investigation in this paper we mainly pursue two goals.
One of them is to give some relatively simple, uniform new proofs with a mild logical flavor for all three above-mentioned classical theorems.
The second goal of the paper is to highlight and elaborate the application of logical tools in the realms of analysis in particular measure theory and make stronger connections between analysis and logic.
In fact, these proofs, beside the fact that are new proofs for some classical theorems which might be of interest of its own right, are indicating
the power of logical methods in measure theory.
Our proofs can be considered as some applications of a particular logical setting, namely "integration logic" (or integral logic), originally introduced in the works of Keisler and Hoover as a setup for dealing with probability and measure spaces by logical means.
Although the methods of the proofs in this paper involves tools from logic,
the reader is not required to be a logician or have any significant amount of knowledge from logic 
to follow the proofs or even apply their ideas for possibly proving similar result.
In fact, 
regarding the logical techniques used in the proofs, one only needs to be able to formalize the problem in 
a certain logical framework by using a very mild language which we explain later. Then, the rest of the logical part of the proof will be automatically handled by a logical machinery in behind, namely the compactness theorem. 
The paper is self-contained and all prerequisites from logic and measure theory are explained in it.

We proceed by explaining more technically about various logical approaches to structures from mathematics in particular analysis.
A usual trend in mathematical logic is to study mathematical objects by 
logical means.
For example, algebraic structures are usually studied using
a classical logical setting called first order logic.
To study structures outside algebra such as measure theory, one usually considers other sorts of logics. 
In fact, there are several ways for incorporating measure and probability in logic.
Probability logics, such as the ones introduced in 
\cite{GoldbringTowsnerAnapproximatelogicformeasures},
\cite{HooverProbabilitylogic}, \cite{KeislerProbabilityquantifiers},
\cite{KuyperTerwijnModeltheorymeasurespacesandprobabilitylogic}, \cite{RaskovicDordeviProbabilityquantifiersoperators} and 
\cite{TerwijnProbabilisticlogicandinduction}, are among various logical frameworks designed to deal with probability and measure structures.
Probability logics have wide range of connections to many other topics such as model theory (see \cite{KeislerProbabilityquantifiers}, \cite{KuyperTerwijnModeltheorymeasurespacesandprobabilitylogic}), combinatorics and dynamical systems (see \cite{GoldbringTowsnerAnapproximatelogicformeasures}), PAC-learning (see \cite{TerwijnProbabilisticlogicandinduction}), etc.
In 
\cite{KeislerProbabilityquantifiers}, an interesting form of probability logics,
called integration logic, is investigated which enables one to use integral operation as a logical quantifier.
In \cite{BagheriPourmahdianThelogicofintegration}, a detailed presentation of
integration logic
is given.
It is worth mentioning that in \cite{MofidiBagheriQuantifiedunivultraproduct}, integration logic is represented as a specific example of a more abstract framework developed with a viewpoint close to functional analysis.
As mentioned earlier, as one of the targets of this paper, we use the framework of integration logic and find concrete applications of it by giving new logical proofs for the three above-mentioned classical theorems.
Using suitable logical frameworks might sometimes enable one to provide uniform proofs with similar techniques for seemingly different theorems. It it indeed the case in this paper.
An advantage of the logical methods of prove we employ in this work is that the results are obtained as the consequences 
(of course not directly, but after putting some additional efforts) 
of a single fundamental fact, namely logical compactness theorem, and that, due to this, the proofs are uniform.
There are many facts in analysis that can be considered in the way that we present in this paper.
But, we are selective and just try to 
reveal the power of some logical
methods in this field.
In fact, methods and strategies of the proofs here (which rely on using of logical compactness theorem) seem to be more general than our results 
and are possibly applicable (to some degree) in various measure existence results.
Although the compactness theorem is stated in
the logic framework, it is easily used in practice.
Compactness theorem in a logical setting roughly states that if we have a family of properties formally stated in that setting and every finite subset of them is satisfied in some structure, then there is a structure that satisfies all of them together.
In fact, there are many interesting real analytic notions or claims that can be decomposed to an infinite number of simple finitary statements. Then, by compactness theorem, the truth of the intended claim is reduced to satisfiability of every finite number of these statements.
Compactness theorem is essentially an existence theorem.
Hence, its applications are so too. In the chapters, we will explain by detail how to use it to prove the mentioned classical existence theorems.

The paper is mostly written for general mathematicians, in particular the people active in analysis or logic as the main audience.
For those readers who are not familiar with logic, we give a general picture of how logic and compactness theorem play role and also roughly explain the steps of the proofs.
The method used is similar to the easy and well-known applications of the classical logical compactness theorem 
in algebra.
For example, in order to show (by a known and old application of the compactness theorem of a classical logical framework called first order logic) that an order relation on a partially ordered set $A$ can be extended to a total order, we first extend the atomic diagram of $A$ (which is, roughly speaking, the set of all first order properties holding in $A$) with the linearity axiom and denote it by $T$. Then, we prove that $T$ is finitely satisfiability (i.e. every finite subset of it is satisfiable in some structure). Then, we apply compactness theorem to obtain a structure $B$ satisfying all expressions in $T$ together. 
But since $B$ satisfies properties mentioned in $T$, it is a linearly ordered extension of $A$. So, in the final step, we push the resulting order from $B$ to $A$. 
In this paper, this procedure is adopted for dealing with measure structures.
But, the arguments are a bit more complicated and require some elementary analytic details. 
More precisely, in the first step, we express (by logical expressions in integration logic) some properties of a measure structure we need to obtain.
These expressions are very close to ordinary ways in mathematics to express properties of measure spaces and form a possibly infinite family of expressions which we call that a theory $T$.
Then in the second step, we prove the finitely satisfiability of $T$, which means that for every finite subset $T'$ of $T$ we find a model of $T'$, which is loosely speaking, a measure structure satisfying every expression in $T'$.
In the third step, we use the logical compactness theorem (which is the main use of logic in this paper) holding in integration logic to conclude from finitely satisfiability that $T$ has a model, which roughly means that there exists a measure structure that satisfies all expressions in $T$.
The above three steps would be enough for the proof of Stone's representation theorem for probability algebras.
But in the proof of Daniell-Stone theorem and Riesz representation theorem, we will need also a forth step 
to induce the measure obtained in the third step on the initial space we started with. This leads us to find a suitable measure on the initial structure as was desired.
Having this explanation 
in mind, the reader who is solely interested in measure theoretic aspects can skip the logical part of the paper and directly goes to the proofs.

Presentation of the paper is as follows. 
In Subsection \ref{prelimmeasureintlog}, 
we briefly review basic measure theoretic concepts
and give a concise introduction to the integration logic. 
Then, In Subsection \ref{SectionpreliminaryLemmas}, we state some preliminary lemmas we will need later in proofs of the main results. 
In Section \ref{applications}, which is the main part of the paper, we give several instances of real analytic notions expressible in integration logic and 
combine them with the power of 
the logical compactness theorem to give new proofs for the mentioned classical measure existence theorems.

\section{Preliminaries}\label{prelim}
\subsection{Preliminaries from measure theory and logic}\label{prelimmeasureintlog}
We review some preliminaries from measure theory.
A (Boolean algebra) measure on a Boolean algebra $\Bc$ of subsets of $M$ is a finitely
additive real-valued function $\mu:\Bc\rightarrow[0,\infty]$ such
that $\mu(\emptyset)=0$ and for any countable sequence $A_k\in \Bc$
of disjoint sets for which $\cup_kA_k\in\Bc$, one has
$\mu(\bigcup_kA_k)=\sum_k\mu(A_k)$.
If $\Bc$ is a $\sigma$-algebra, $\mu$ is called a measure.
Now we recall the definition of subspace measures. 
For any measure space $(M,\Bc,\mu)$, the outer measure $\mu^*$ on $M$ is defined by
$\mu^*(X)=\inf\{\mu(A)|\ X\subseteq A\in\Bc\}$ for every $X \subseteq M$.
If $N\subseteq M$, then $\Bc_N=\{A\cap N|\ A\in \Bc\}$ forms a
$\sigma$-algebra of subsets of $N$ and the restriction of
$\mu^*$ to it, denoted by $\mu_N$, is a measure.
Indeed, elements of $\Bc_N$ are $\mu^*|_N$-measurable.
$\mu_N$ is called the {\em subspace measure} on $N$.
If $f:M\rightarrow \Rn$ is measurable, by
$\int_N f|_N$ is meant $\int_N (f|_N)d\mu_N$ where $f|_N$ is the notation for restriction of $f$ to $N$.

\begin{pro}\label{Fremlin subspace1} \emph{(see \cite{FremlinMeasuretheoryvol2}, Subsection 214)}
Let $(M,\Bc,\mu)$ be a measure space, $N\subseteq M$ and $f$ an
integrable function on $M$. Then the followings hold.

(a) $f|_N$ is $\mu_N$-integrable, and $\int_N f\leqslant\int f$ if $f$ is
nonnegative.

(b) If either $N$ is of full outer measure in $M$ or $f$ is zero
almost everywhere on $M-N$, then $\int_N f|_N=\int_Mf$.
\end{pro}

The following theorem is a very useful method for constructing the measures by extension.

\begin{thm}\label{caratheodoryextensiontheorem} \emph{(Carath\'eodory extension theorem)}
Let $\mu$ be a finite measure on a Boolean algebra $\mathcal{B}$ of subsets of
$M$. Then $\mu$ has an extension $\bar\mu$ to $\sigma(\mathcal{B})$ (the $\sigma$-algebra generated by $\mathcal{B}$).
Moreover the extension constructed in this way is unique.
\end{thm}

Carath\'eodory's construction is usually divided into two parts.
First, one extends any measure $\mu$ on $(M,\mathcal{B})$ to an outer measure
by defining
$\mu^*(E)=\inf\big\{\sum_k\mu(A_k):\ \ E\subseteq\bigcup_{k < \omega}A_k, \ \ A_k\in \mathcal{B} \ \big\}$
for each $E\subseteq M$.
Then, one defines a complete measure on $M$ by restricting $\mu^*$ to
$\mu^*$-measurable sets, i.e. sets $A\subseteq M$ for which
$\mu^*(E)=\mu^*(E\cap A)+\mu^*(E-A)$ for every $E\subseteq M.$
Then, every sets in $\sigma(\mathcal{B})$ is $\mu^*$-measurable. Carath\'eodory extension
is in particular used to define the product measures $\mu^n$ on $M^n$.
The domain of $\mu^n$ is the the smallest $\sigma$-algebra $\Bc^n$ generated by
the rectangles $A_1\times\ldots\times A_n$ where each $A_i$ is measurable.

Diagonal sets are usually non measurable in the product measure. Since these sets are important in logic, one tries to add them as measurable sets. Let $\Bc^{(n)}$ be the $\sigma$-algebra generated by $\Bc^n$ and the diagonal subsets of $M^n$.

\begin{pro} \label{diagonal} \emph{(see \cite{KeislerProbabilityquantifiers})}
Let $(M,\Bc,\mu)$ be a measure space of finite measure such that
every singleton is measurable. Then, there is a unique measure
$\mu^{(n)}$ on the $\sigma$-algebra $\Bc^{(n)}$ of subsets of $M^n$
generated by the $n$-fold rectangles and the diagonals $D_{ij}^n$
which extends $\mu^n$ and such that for any $i\neq j$,
$\mu^{(n)}(D_{ij}^n)=\sum_{x\in M}\mu(\{x\})^2.$
Moreover, for any $X\in\Bc^{(n)}$, there is a $\mu^n$-measurable
set $U$ such that $\mu^{(n)}(X\Delta U)=0$.
\end{pro}

Now we get into the logic and briefly review the framework of "integration logic",
investigated in \cite{BagheriPourmahdianThelogicofintegration}, \cite{HooverProbabilitylogic} and \cite{KeislerProbabilityquantifiers}
for studying measure and probability structures by logical means.
We use the terminology of \cite{BagheriPourmahdianThelogicofintegration}.
Using this framework enables us to formalize and express certain measure theoretic properties of spaces, functions on them, etc, in a unified way.
We first quickly review essential concepts and then formally define some notions.
By a \textit{simple (measure) relational structure} (or simply, a structure) in this paper, 
intuitively we mean a measure space we wish to study equipped with a family of \textit{relations}, where by a relation we mean a real-valued measurable function on (some power of) the measure space. Also there might be a family of elements of the ambient set of the measure space which are needed to be considered as distinguished elements.
We usually assign a symbol corresponding to each of those relations and distinguished elements and call them \textit{relation symbols} and \textit{constant symbols} respectively.
We also call the set of such symbols a (relational) \textit{language} and usually denote it by $\mathcal{L}$.
Furthermore, we call the structure to which the symbols of $\mathcal{L}$ is referring, and more generally, any other structure in which the symbols in $\mathcal{L}$ are interpreted, a \textit{$\mathcal{L}$-structure} (as formally will be defined in Definition \ref{defLstructure}).
In fact, in order to 
systematically study one or a family of structures by logical means, we usually first choose a suitable language $\mathcal{L}$ consisting of the symbols corresponding to all relations and distinguished elements we intend to deal with or investigate 
in our structure(s). So now those structure(s) can be seen as $\mathcal{L}$-structure(s).
Then, we can use symbols in $\mathcal{L}$ as well as \textit{variable symbols} (as are defined below) and \textit{logical symbols} (namely, connectives and quantifiers as explained below) to write formal logical expressions called formulas, statements and sentences, describing our $\mathcal{L}$-structure(s) logically.
This enables us to study mathematical properties of the structure(s) in hand through formal logical tools and syntactic methods.
Note that in most of the structures in this paper, real functions on the spaces play the main role.
In particular, in the case of Daniell-Stone and Riesz representation theorems, one has to deal mainly with just certain spaces of real-valued functions on ambient sets and some functions (more precisely, functionals) on those spaces.
Due to this reason, it would be sufficient for us to only work with relational structures and even more, let our languages contain, beside possibly some constant symbols, solely unary-relational symbols (i.e. real-valued functions on the ambient space itself and not a power of it), with the intention to be interpreted as the functions belonging to those spaces.
Because of this, in reviewing the setup of integration logic in this paper, we restricted ourselves to only relational structures and languages.
However, it worth to mention that in general, in integration logic (and more generally in mathematical logic), languages can contain, in addition, another type of symbol namely \textit{function symbols} with the intended interpretation as functions from the structure (or some power of it) to itself.

We always assume that a language $\mathcal{L}$ contains a distinguished binary
relation symbol $\e$ for equality.
We also assume that to each relation symbol $R$ is assigned a
nonnegative real number $\flat_R$ called its \textit{universal bound}.
In particular, $\flat_{\e}= 1$.
As will be explained more in Definition \ref{defformulasinintegrallogic}, logical symbols consist of the binary functions $+,\ \cdot$, the
unary absolute value function $| \ \ |$ and a $0$-ary function $r$ for each real number $r$. These functions are considered as connectives.
The integration symbol $\int$ is also a logical symbol and used as a quantifier.
We also use an infinite list $x,y,...$ of individual variable symbols.
We call the family of all variable symbols and constant symbols, the collection of $\mathcal{L}$-\textit{terms}.

\begin{dfn}\label{defLstructure}
{\em Let $\mathcal{L}$ be a relational language. A {\em simple (relational)
$\mathcal{L}$-structure} (or simply, a $\mathcal{L}$-structure) is a non-empty measure space $(M,\Bc,\mu)$, in
which every singleton is measurable and $\mu(M)=1$, 
equipped with:

- for each $n$, the measure $\mu^{(n)}$ given by Proposition \ref{diagonal}

- for each constant symbol $c\in \mathcal{L}$ (if there is any), an element $c^M\in M$

- for each $n$-ary relation symbol $R\in \mathcal{L}$ (if there is any), a measurable
function $R^M:M^n\rightarrow\Rn$ such that $|R^M(\a)|\leqslant \flat_R$
for any $\a\in M$.

We refer to $R^M$ and $c^M$ as the \textit{interpretations} (in $M$) of the relation and constant symbols $R$ and $c$.
}
\end{dfn}

Note that in every structure, the binary equality relation $\e(x,y)$ is interpreted as a two variable function taking value $1$ if 
$x=y$
and $0$ otherwise. 
For a language $\mathcal{L}$, the family of $\mathcal{L}$-formulas is inductively defined as follows.

\begin{dfn}\label{defformulasinintegrallogic}{\em
\begin{enumerate}
\item{If $R$ is a $n$-ary relation symbol in  $\mathcal{L}$ and $t_1,...,t_n$ are
 $\mathcal{L}$-terms, then $R(t_1,...,t_n)$ is  
a formula. 
In particular, $\e(x,y)$ is 
a formula.}

\item{For any $r\in\Rn$,\ $r$ is a formula.} 

\item{If $\phi$ and $\psi$ are formulas then $|\phi|, \phi+\psi$
and $\phi\cdot\psi$ are formulas.}

\item{If $\phi(\x,y)$ is a formula, then $\int\phi(\x,y)dy$
is a formula.} 
\end{enumerate}}
\end{dfn}

It is important to note that the expressions $\phi\vee\psi$ and $\phi\wedge\psi$ (the max and min of two formulas $\phi$ and $\psi$) are also formulas since they can be built using $+$, $-$ and $| \ |$.
In fact we have
$\phi\vee \psi=\frac{\phi+\psi+|\phi-\psi|}{2}$
and
$\phi \wedge \psi=\frac{\phi+\psi-|\phi-\psi|}{2}$.
\textit{Free variables} of formulas are easily defined (by induction) as the variables which are not bounded by the quantifies $\int$.
For example in the formula $\int (x+y) \ dy+ |2z|$, the variables $x$ and $z$ are free while $y$ is bounded by the quantifier $\int$.
One writes $\phi(x_1,...,x_n)$ to indicate that all free variables
of the formula $\phi$ appear in $x_1,...,x_n$. 
A \textit{closed} formula is a formula
without free variables. If $\phi(\x)$ is a formula and
$\a\in M^{|\x|}$, the value of $\phi(\a)$ in $M$, denoted by
$\phi^M(\a)$, is defined inductively in the natural way.
For example $$(\phi+\psi)^M(\a)=\phi^M(\a)+\psi^M(\a), \ \ \ \ \ \
\Big(\int\phi(\x,y) dy\Big)^M(\a)=\int_M \phi^M(\a,y)dy.$$
So $\phi(\x)$ gives rise to a real-valued function on $M^{|x|}$, which is called the \textit{interpretation} of the formula $\phi$ and is denoted by $\phi^M$.
Note that, in particular, if $\phi$ is a closed formula, then for any model $M$, $\phi^M$ is uniquely determined and is a real number.
For example if $\phi=\int\psi(y)dy$ where $\psi(y)$ is a formula, then $\phi^M=\int_M \psi^M(y)dy$.
A \emph{statement} is an expression of the form $\phi(x) \geqslant r$ or $\phi(x)= r$ for some formula $\phi(x)$ and some $r \in \mathbb{R}$.
If $\phi$ is a closed formula, then the statement
is called a \emph{closed statement} (or \textit{sentence}).
Any set of closed statements is called a \emph{theory}.
Obviously expressions such as $\phi(x) \leqslant r$, $\phi(x) \geqslant \psi(x)+r$ or $\phi(x)= \psi(x)+r$, where $\phi(x)$ and $\psi(x)$ are formulas, are also statements since they can be written in the form $-\phi(x) \geqslant -r$, $\phi(x)-\psi(x) \geqslant r$ or $\phi(x)-\psi(x)=r$ while $-\phi(x)$ and $\phi(x)-\psi(x)$ are again formulas.
A closed statement $\phi=r$ or $\phi \geqslant r$ is \textit{satisfied} in a simple $\mathcal{L}$-structure $M$, denoted by $M\vDash "\phi=r"$ and $M\vDash "\phi \geqslant r"$, if $\phi^M=r$ and $\phi^M \geqslant r$ respectively.
A simple $\mathcal{L}$-structure $M$ is a 
\textit{model} of a theory $T$, denoted $M\vDash T$, if each of its statements is satisfied in $M$.
A theory is satisfiable if it has a model.
A theory is finitely satisfiable if every finite subset of it has a model.
The theory of a structure $M$ is the collection of statements satisfied in it.
Such theories are called complete.

As some examples of basic measure theoretic properties expressible in integration logic, one can mention that the expression "singletons have measure zero" is stated by
$\int\e(x,x)=0$. Also the expression "the space has total measure $r$" is written by $\int 1\ dx=r$. 
To see more examples 
the reader can see Remark \ref{somesampleexpresibleproperties}.

\vspace{1mm}

Since now on, we work with 
The main logical tool used in this paper is the following theorem which is basically Theorem 4.7 of \cite{BagheriPourmahdianThelogicofintegration}.

\begin{thm}\label{logicalcompactnesstheorem} \emph{(Logical compactness theorem)}
Any finitely satisfiable theory is satisfiable.
\end{thm}

Now we want to mention some technical points about the above theorem. However, this paragraph is independent of the rest of the paper and the one who is not interested in logical details, can skip that.
It worth to be noted that in papers \cite{BagheriPourmahdianThelogicofintegration} and \cite{KeislerProbabilityquantifiers}, in addition of the notion of simple $\mathcal L$-structure, a more general notion of structure, namely graded structures, was defined and compactness theorem was proved for such more general structures too. Simple $\mathcal L$-structures which are the concern of this paper are special instances of graded ones. 
In fact, Theorem \ref{logicalcompactnesstheorem} is the compactness theorem restricted to the class of simple structures (with just unary relations with exception of equality), as it is the concern of this paper. This is a very applicable variant of compactness theorem with many applications in different situations.

We recall a lemma which helps us to simplify the arguments.
We say that a theory $T$ is \emph{finitely approximately satisfiable} if every finite list of closed statements of the form $\phi=r$ and $\phi \geqslant r$ in $T$ is \emph{approximately satisfiable} which means that for every $\epsilon>0$, there exists a model $M$ which $\epsilon$-approximately (with error at most $\epsilon$) satisfies that finite list of statements, or more precisely, $|\phi^M-r|\leqslant\epsilon$ and $\phi^M \geqslant r-\epsilon$ respectively for each of such statements.
Using a standard technique by a non-principal ultrafilter on $\mathbb N$, one can easily show that:

\begin{lem}\label{approximate}
If $T$ is finitely approximately satisfiable then it is finitely satisfiable.
\end{lem}

Section \ref{applications} is where we use compactness theorem
in order to prove the well-known classical measure existence theorems in measure theory we mentioned before.

\subsection{Some preliminary lemmas}\label{SectionpreliminaryLemmas}
In this subsection, we state and prove a few statements which will be used in the proof of the main results in the next section.
We denote the characteristic function of a set $U$ by $\chi(U)$.
By a one-side (both-sides) unbounded interval in $\mathbb{R}$ we mean an interval which is unbounded from one of right or left sides (both sides).
Through this subsection, we assume that $\mathcal{A}$ ia a vector lattice of real-valued functions on a set $X$
containing the function $1_X$ (the function with value $1$ on every $x \in X$). Also we assume that $I$ is a positive linear real-valued function on $\mathcal{A}$.

\begin{lem}\label{tendtochar}
	Let $f_1,\ldots,f_m$ be a not necessarily distinct finite family of functions in $\mathcal{A}$ and $U_1,\ldots,U_m$ a not necessarily distinct family of open 
	intervals in $\mathbb{R}$, either bounded, one side unbounded or both sides unbounded.
	Then, there is some increasing sequence of $[0,1]$-valued functions in $\mathcal{A}$ tending pointwise to
	$\chi(\bigcap_{i=1}^m f_i^{-1}(U_i))$ while the supports of functions in the sequence are subsets 
	of $\bigcap_{i=1}^m f_i^{-1}(U_i)$, and similarly, 
	there is some increasing sequence of $[0,1]$-valued functions in $\mathcal{A}$ tending pointwise to
	$\chi(\bigcup_{i=1}^m f_i^{-1}(U_i))$ while the support of its functions are subsets 
	of $\bigcup_{i=1}^m f_i^{-1}(U_i)$.
	Moreover, the statement holds when we replace the words "open intervals" to "closed intervals", "increasing" to "decreasing" and "subset" to "superset".
	In this case, in particular, $U_i$'s can be single real numbers since every real number can be seen as a closed interval.
\end{lem}
\begin{proof}
We first start to prove the lemma for just one $f$ and one $U$.
Fix $f \in \mathcal{A}$ and $\alpha\in\mathbb{R}$. For every $n \in \mathbb{N}$ set
$$g^1_n(x):=n\big(\min(f(x),\alpha+\frac{1}{n})-\min(f(x),\alpha)\big) \ \ \ , \ \ \ g^2_n(x):=n\big(\max(f(x),\alpha)-\max(f(x),\alpha-\frac{1}{n})\big)$$
$$h^1_n(x):=n\big(\min(f(x),\alpha)-\min(f(x),\alpha-\frac{1}{n})\big) \ \ \ , \ \ \ h^2_n(x):=n\big(\max(f(x),\alpha+\frac{1}{n})-\max(f(x),\alpha)\big).$$
Then, it is not very difficult to see that the sequences of functions $(g^1_n)_{n<\omega}$ and $(g^2_n)_{n<\omega}$ increase pointwise to $\chi(\{x:\alpha < f(x)\})$ and $\chi(\{x: f(x) < \alpha\})$ respectively.
Also support of each $g^1_n$ and each $g^2_n$ is a subset of $\{x:\alpha < f(x)\}$ and $\{x: f(x) < \alpha\}$ respectively.
Similarly, $(h^1_n)_{n<\omega}$ and $(h^2_n)_{n<\omega}$ decrease pointwise to
$\chi(\{x:\alpha \leqslant f(x)\})$ and $\chi(\{x: f(x) \leqslant \alpha\})$ respectively while support of each $h^1_n$ and each $h^2_n$ is a superset of $\{x:\alpha \leqslant f(x)\}$ and $\{x: f(x) \leqslant \alpha\}$ respectively.
Also $g^1_n$'s, $g^2_n$'s, $h^1_n$'s and $h^2_n$'s are functions in $\mathcal{A}$ taking values in $[0,1]$.
So the statement is proved for one $f$ and one $U$ of the form $(-\infty,\alpha), (\alpha,\infty), (-\infty,\alpha],$ and $[\alpha,\infty)$.

For $\alpha<\beta$, if $(g_n)_{n<\omega}$ and $(h_n)_{n<\omega}$ are sequences of functions obtained above increasing to $\chi(\{x: f(x)<\beta\})$ and $\chi(\{x: \alpha< f(x)\})$ respectively,
then $(g_n\wedge h_n)_{n<\omega}$ increases to
$\chi(\{x: \alpha< f(x)<\beta\})$ and the support of each $g_n\wedge h_n$ is a subset of $\{x: \alpha< f(x)<\beta\}$.
Similarly, if $(g_n)_{n<\omega}$ and $(h_n)_{n<\omega}$ are sequences of functions obtained above decreasing to
$\chi(\{x: f(x)\leqslant\beta\})$ and $\chi(\{x: \alpha\leqslant f(x)\})$ respectively, then $(g_n\wedge h_n)_{n<\omega}$ decreases to
$\chi(\{x: \alpha\leqslant f(x)\leqslant\beta\})$ and the support of each $g_n\wedge h_n$ is a superset of $\{x: \alpha\leqslant f(x)\leqslant\beta\}$.
These prove the statement of lemma for one $f$ and one $U$ of the form $(\alpha,\beta)$ and $[\alpha,\beta]$.
It worth mentioning that if $(g_n)_{n<\omega}$ is a sequence which increases to $\chi(\{x: f(x)<\beta\})$, then
$(1-g_n)_{n<\omega}$ decreases to $\chi(\{x:\beta\leqslant f(x)\})$ and vice versa.

Let $U_1,\ldots,U_m$ be open intervals of the above forms and $f_1,\ldots,f_m \in \mathcal{A}$.
Assume that for each $i=1,\ldots,m$, the sequence $(g^i_n)_{n<\omega}$ is the sequence of functions increasing to $\chi(\{x:f(x) \in U_i\})$ obtained in the way that explained above.
Then the sequences of $[0,1]$-valued functions $(g^1_n \wedge \ldots \wedge g^m_n)_{n<\omega}$ and $(g^1_n \vee \ldots \vee g^m_n)_{n<\omega}$ increase to 
$\chi(\bigcap_{i=1}^m f_i^{-1}(U_i))$ and $\chi(\bigcup_{i=1}^m f_i^{-1}(U_i))$ respectively.
Moreover, since the support of each $g^i_n$ is a subset of $\{x:f(x) \in U_i\}$, then the supports of $(g^1_n \wedge \ldots \wedge g^m_n)_{n<\omega}$ and $(g^1_n \vee \ldots \vee g^m_n)_{n<\omega}$ are subsets of $\bigcap_{i=1}^m f_i^{-1}(U_i)$ and $\bigcup_{i=1}^m f_i^{-1}(U_i)$ respectively.
Similarly, the same statements hold when we replace being open by being closed for $U_i$'s, increase to decrease and subset by superset.
In particular, in this case one sees that the proof works when some of $U_i$'s are single real numbers since every real number can be seen as a closed interval.
\hfill $\square$
\end{proof}

\begin{rem}\label{DecSeqCores}
{\emph{Let $f \in \mathcal{A}$ and $\alpha\in \mathbb{R}$.
Also let $(h_n)_{n <\omega}$ be the pointwise decreasing sequences of functions obtained in the proof of Lemma \ref{tendtochar} tending to $\chi(f^{-1}(\{\alpha\}))$.
Then, we call $(h_n)_{n <\omega}$ the \emph{"decreasing sequence corresponding to $\alpha$ for $f$"}.
We remind from Lemma \ref{tendtochar} that $h_n$'s take values in $[0,1]$.
Therefore, since $I$ is positive linear, for every $n$, $0 \leqslant I(h_n)\leqslant I(1_X)$.
So, $(I(h_n))_{n <\omega}$ is a sequence of non-negative real numbers
and since $(h_n)_{n <\omega}$ is decreasing, again by positive linearity of $I$,
the sequence $(I(h_n))_{n <\omega}$ is decreasing (but not necessarily strictly decreasing). 
Hence, $\lim_{n \rightarrow \infty} I(h_n)$ exists and is a non-negative real number.
We call $\alpha$ an \emph{"inessential value of $f$ with respect to $I$"} if $\lim_{n \rightarrow \infty} I(h_n)=0$.
It is not hard to see that for every $n$ and every $x$ outside of $f^{-1}(\alpha-\frac{1}{n},\alpha+\frac{1}{n})$, we have $h_n(x)=0$.
So, the support of each $h_n$ is a subset of $f^{-1}(\alpha-\frac{1}{n},\alpha+\frac{1}{n})$.
}}
\end{rem}

\begin{lem}\label{nullpoints}
Let $f_1,\ldots,f_t\in\mathcal{A}$ and $(r,s)$ be an interval in $\mathbb{R}$. Also assume that $I(1_X) \not=0$.
Then, there exists a $\alpha\in(r,s)$ which is an inessential value of every $f_i$ with respect to $I$.
\end{lem}
\begin{proof}
We prove the claim for $t=1$ namely for one $f$.
The general case is similar but needs some more effort. 
Assume for contradiction that there is no such $\alpha$ for $f$.
For each $u\in(r,s)$, let $(h^u_n)_{n < \omega}$ be the decreasing sequence corresponding
to $u$ for $f$ (defined in Remark \ref{DecSeqCores}) and let $S(u):=\lim_{n \rightarrow \infty} I(h^u_n)$.
So, for every $u\in(r,s)$, since $u$ is not an inessential value of $f$ with respect to $I$, we have $S(u)>0$ .
Thus, for some $m\in\mathbb{N}$, there exists an infinite subset $V$ of $(r,s)$
such that $S(u)>\frac{I(1_X)}{m}$ for each $u\in V$. Let $v_1,\ldots, v_{2m}\in V$ be distinct.
So in particular $S(v_i)>\frac{I(1_X)}{m}$ for each $i=1,\ldots,2m$.
By using Remark \ref{DecSeqCores}, for each $i=1,\ldots,2m$ and $n$, support of each $h^{v_i}_n$ is a subset of $f^{-1}(v_i-\frac{1}{n},v_i+\frac{1}{n})$.
Thus, for each $i=1,\ldots,2m$, we can choose a function $g_i$ from the sequence of functions $(h^{v_i}_n)_{n < \omega}$ in such a way that at the end of the selections, the supports of selected $g_i$'s are mutually disjoint.
Since $h^{v_i}_n$'s take values in $[0,1]$ (by Remark \ref{DecSeqCores}), for every $i$ we have $0 \leqslant g_i(x) \leqslant 1$ (for every $x$).
Now the function $g:=\sum_{i=1}^{2m}g_i$ belongs to $\mathcal{A}$ and $0 \leqslant g(x) \leqslant 1$ for every $x$.
We remind that $I$ is positive linear. Therefore, $0 \leqslant I(g) \leqslant I(1_X)$. On the other hand, for each $i$, since $(h^{v_i}_n)_{n < \omega}$ is a decreasing sequence, by positive linearity of $I$, the sequence $(I(h^{v_i}_n))_{n < \omega}$ is decreasing which follows that $I(g_i) \geqslant \lim_{n \rightarrow \infty} I(h^{v_i}_n)=S(v_i)$. So we have
$I(g)=\sum_{i=1}^{2m}I(g_i) \geqslant \sum_{i=1}^{2m} S(v_i) \geqslant 2m.\frac{I(1_X)}{m}=2I(1_X)$.
Combining above facts, we have $I(1_X)=0$ which is a contradiction.
\hfill $\square$
\end{proof}

\vspace{1.5mm}

We say that a family of sets covers a set $X$ (or is a covering of $X$) if their union contains $X$ as a subset.

\begin{lem}\label{changecover}
	Let $(Y,\Bc,\mu)$ be a measure space of finite measure, $\mathcal{K}$
	a $\mathbb{R}$-vector lattice of real measurable functions on $Y$ and
	$\mathcal{C}$ the Boolean algebra generated by sets $f^{-1}(0,\infty)$
	where $f\in\mathcal{K}$.
	Let $X \subseteq Y$ and $\{U_n\}_{n<\omega} \subseteq \Cc$ be a covering of $X$.
	Then, for each $\epsilon>0$, there is a covering $\{V_n\}_{n<\omega}$ of $X$ such
	that (i): for each $n$, $V_n=f_n^{-1}(0,\infty)$ for some $f_n\in \mathcal{K}$, (ii): sum
	of the measures of members of this covering does not differ from the sum of the measures of members of the 
	first covering $\{U_n\}_{n<\omega}$ with more than $\epsilon$, and (iii): 
	$\mu(f_n^{-1}(\{0\}))=0$ for each $f_n$ mentioned above.
\end{lem}

\begin{proof}
	For convenience, we call a measurable subset $D$ of $Y$ a type I subset if there is some $f \in \mathcal{K}$ such that $D=f^{-1}(0,\infty)$.
	We call $D$ a nice type I subset if there exists such $f$ with the additional property that $\mu(f^{-1}(\{0\}))=0$.
	Similarly, we call a measurable subset $D$ a type II subset if there is some $f \in \mathcal{K}$ such that $D=f^{-1}[0,\infty)$.
	Since $(f^{-1}(0,\infty))^c=(-f)^{-1}[0,\infty)$, by using disjunctive normal form representation of members of Boolean algebras, every element of $\mathcal{C}$, such as $U_n$'s, can be represented as a
	finite union of finite intersections of type I or type II sets.
	We call such representation of any $U_n$ a good representation of it.
	
	Let $\epsilon>0$ be fixed. Also for each $n$, fix a good representation of $U_n$. For each $n$, let $U'_n$ be the modification of $U_n$ 
	by replacing the clauses of
	the form $f^{-1}[0, \infty)$ in the mentioned good representation of $U_n$ by some bigger sets
	$f^{-1}(-\delta_n,\infty)$ for some small enough positive real numbers $\delta_n$'s,
	in such a way that the difference between the sum of the measures of $U'_n$'s, namely $\sum_{n<\omega}\mu(U'_n)$,
	and that of $U_n$'s, namely $\sum_{n<\omega}\mu(U_n)$, is not more than $\epsilon$.
	It is easily seen that the family of $U'_n$'s is a covering of $X$.
	We note that for every $f \in \mathcal{K}$ and real number $\delta$, we have $f^{-1}(-\delta,\infty)=f'^{-1}(0,\infty)$ where $f'=f+\delta$. Since $\mathcal{K}$ is a vector lattice, obviously $f' \in \mathcal{K}$.
	Similarly, for every $f_1,f_2 \in \mathcal{K}$, we have $f_1^{-1}(0,\infty) \cap f_2^{-1}(0,\infty)=g^{-1}(0,\infty)$ and $f_1^{-1}(0,\infty) \cup f_2^{-1}(0,\infty)=h^{-1}(0,\infty)$ where
	$g= f_1 \wedge f_2$ and $h=f_1 \vee f_2$ respectively. Again, since $\mathcal{K}$ is a vector lattice, clearly $g,h \in \mathcal{K}$.
	Now using these facts,
	it is not hard to see that for each $n$, we have $U'_n=f_n^{-1}(0,\infty)$ for some $f_n \in \mathcal{K}$.
	So $U'_n$'s are type I subsets of $Y$.
	Let denote the family of $U'_n$'s by $\mathcal{U}$.
	We remind that $\mathcal{U}$ is a covering of $X$.
	If all $U'_n$'s are nice type I subsets, then we are done. Otherwise, $I \not = \emptyset$ where $I \subseteq \mathbb{N}$ is the set of all $n \in \mathbb{N}$ such that $U'_n$ is not nice
	(so, for each $n \in I$, we have $\mu(f_n^{-1}(\{0\}))>0$).
	In that case, in the following procedure, we will replace some members of $\mathcal{U}$ with some families of subsets of $Y$ in such a way that after these replacements, our new $\mathcal{U}$
	still remains a covering of $X$,
	the sum of the measures does not differ with more than $\epsilon$ and moreover, our new $\mathcal{U}$ only contains nice type I subsets of $Y$.
	The procedure is as follows.
	
	Corresponding to each $n \in I$, we can find a sequence $a^n_1,a^n_2,\ldots$ of distinct positive real numbers decreasing to 0
	such that the followings hold.
	
	\vspace{2mm}
	
	(i) $\mu(f_n^{-1}(\{a^n_j\}))=0$ for each $j<\omega$.
	
	\vspace{2mm}
	
	(ii) $\mu(U''_n)$ does not differ from $\mu(U'_n)$ with more than $\frac{\epsilon}{4^n}$, where $U''_n:=f_n^{-1}(a^n_2,\infty)$,
	
	\vspace{2mm}
	
	(iii) $\sum_{j=1}^{\infty} \mu(U''_{n,j}) \leqslant \frac{\epsilon}{4^n}$, where for every $j<\omega$, $U''_{n,j}:=f_n^{-1}(a^n_{j+2},a^n_j)$,
	
	\vspace{2mm}

	It is easy to see that for each $n \in I$, $U''_n=p_n^{-1}(0,\infty)$ where $p_n:=f_n-a^n_2$.
	Also for $j<\omega$, we have $U''_{n,j}=f_n^{-1}(a^n_{j+2},a^n_j)=(h_{n,j})^{-1}(0,\infty)$ where $h_{n,j}:=(f_n-a^n_{j+2}) \wedge (a^n_j-f_n) \in \mathcal{K}$.
	So $U''_{n,j}$'s and $U''_n$'s are all type I subsets of $Y$. Also (i) guarantees that
	they are all nice type I subsets.
	Now for each $n \in I$, we remove the member $U'_n$ from the family $\mathcal{U}$ and instead, add $U''_n$ and all $U''_{n,j}$ (for each $j<\omega$) to $\mathcal{U}$.
	Now, this new $\mathcal{U}$ is a family of nice type I subsets of $Y$.
	Furthermore, for each $n \in I$, we have
	$$\big(\bigcup_{j<\omega} U''_{n,j}\big) \bigcup U''_n=\big(\bigcup_{j<\omega} f_n^{-1}(a^n_{j+2},a^n_j)\big) \bigcup f_n^{-1}(a^n_2,\infty)=f_n^{-1}(0,\infty)=U'_n.$$ 
	So, having this, it is not hard to see that $\mathcal{U}$ is still a countable covering of $X$.
	Also (ii) and (iii) guarantee that for each $n \in I$, sum of the measures of $U''_n$ and all $U''_{n,j}$'s does not differ from $\mu(U'_n)$ with more than $\frac{\epsilon}{2^n}$.
	So, sum of measures of members of our new $\mathcal{U}$ does not differ from sum of measures of $U_n$'s with more than $\epsilon$.
	\hfill $\square$
\end{proof}

\section{Logical compactness theorem and new proofs for some classical measure existence theorems}\label{applications}
Existence theorems appear in many branches of mathematics.
Logical compactness theorem (Theorem \ref{logicalcompactnesstheorem}) is itself an existence theorem and in fact
a fundamental one.
As we will see, it can be used in measure theory in a systematic way
to give relatively easy uniform proofs for many measure existence theorems.
Meantime, some interesting mathematical theories are axiomatized in the setting of integration logic.
To start with, in the following remark, we mention some basic measure theoretic properties expressible in this setting.
\begin{rem}\label{somesampleexpresibleproperties}
The expression
"the space has total measure $1$" is stated by
$\int 1\ dx=1$. 
Also for any formulas $\phi(x)$ and $\psi(x)$ with the same free variables $x$ (in a relevant language in the integration logic as defined before), the expressions 
"$\phi(x)=0$ almost everywhere" and "$\phi(x)=\psi(x)$ almost everywhere" are stated by the closed statements $\int\ |\phi(x)| \ dx=0$ and 
$\int\ |\phi(x)-\psi(x)| \ dx=0$ in integration logic, where we remind that the interpretation of $\phi$ and $\psi$ are measurable functions on our measure space.
Let $A=\{r_1,\ldots,r_n\}$ be any finite subset of $\Rn$.
Then, the similar expression "range$(\phi) \in A \ (a.e)$",
is expressible by the closed statement
$$\int |(\phi(x)-r_1).(\phi(x)-r_2)\ldots(\phi(x)-r_n)|dx=0.$$
\end{rem}

\subsection{Stone's representation theorem for probability algebras}
The first applications of the logical compactness theorem we present in this paper is a new proof for the Stone's representation theorem for probability algebras.
Recall that a Boolean algebra is $\sigma$-complete if every countable non-empty subset $a_1,a_2,\ldots$ of it
has a least upper bound $\vee_i a_i$ (or $\sup_{i<\omega}a_i$)
and a greatest lower bound $\wedge_i a_i$ (or $\inf_{i<\omega}a_i$).
A measure algebra
(see for example \cite{FremlinMeasuretheoryvol3} Definition 321A)
is a $\sigma$-complete Boolean algebra
$(B,\wedge,\vee,\ ',\textbf{0},\textbf{1})$ equipped with a
map $\mu:B\rightarrow[0,\infty]$ such that
(i) $\mu(a)=0$ if and only if $a=\textbf{0}$, and (ii) if $a_1,a_2,\ldots$ are pairwise disjoint (i.e. $a_i \wedge a_j=\textbf{0}$ for every distinct $i$ and $j$), then $\mu(\vee_i a_i)=\sum_i\mu(a_i)$.
Note that the notations $\wedge$, $\vee$ and $'$ in here stand for their corresponding operations 
in the Boolean algebra $B$ and shouldn't be confused with the logical connectives defined before (with the same notations $\wedge$ and $\vee$) which stand for the "max" and "min" of two logical formulas.
If $\mu(\textbf{1})=1$, the measure algebra is called a probability algebra.
A $\sigma$-order-continuous isomorphism (or sequentially order-continuous isomorphism) (see \cite{FremlinMeasuretheoryvol3} Definition 313H) between measure algebras $B_1,B_2$ is a measure preserving Boolean isomorphism
$\phi:B_1\rightarrow B_2$ such that $\phi(\vee_i a_i)=\vee_i \phi(a_i)$ for every increasing
sequence $a_1,a_2,...$ in $B_1$.
We recall that in any Boolean algebra, a partial order relation $\leqslant$ is naturally defined by $a \leqslant b$ if and only if $a \wedge b =a$.

To every probability space $(M,\mathcal{A},\bar\mu)$ is associated
a probability algebra as follows.
Say $X_1, X_2 \in \mathcal{A}$ are equivalent if their symmetric difference is null.
The equivalence class of $X$ is denoted by $[X]$.
Then the set of equivalence classes forms a Boolean algebra in the natural way and
$\mu([X])=\bar\mu(X)$ makes of it a probability algebra.

A classical proof for Stone's representation theorem for measure algebras could be found in for example \cite{FremlinMeasuretheoryvol3}(321J). In the following, for simplicity, we prove the theorem for probability algebras. But it is easy to see that a slight modification of the following proof gives rise to a proof for general finite measure algebras.

\begin{thm}
\emph{(Stone's representation theorem for probability algebras)}
Let $(B,\mu)$ be a probability algebra. Then, there is a probability space $(M,\mathcal{B},\bar\mu)$
whose associated probability algebra is $\sigma$-order-continuous isomorphic to $(B,\mu)$.
\end{thm}
\begin{proof}
Let $\mathcal L$ be a language (as defined in Subsection \ref{prelimmeasureintlog}) consisting of a unary relation symbol $R_a$ with universal bound $1$ for each $a\in B$. Let $T$ be a $\mathcal{L}$-theory consisting of the following expressions (axioms) which can be carefully written as some closed statements in integration logic in the language $\mathcal{L}$ (one can get help from Remark \ref{somesampleexpresibleproperties} for stating them).

\begin{enumerate}
\item{$R_a(x) \stackrel{a.e}{=} 0$ or $1$ \ \ \ \ \ \ (for each $a \in B$),}\label{StoneprobalgaxiomRa0or1}

\item{$\int R_a(x)dx=\mu(a)$ \ \ \ \ \ \ (for each $a \in B$),}\label{Stoneprobalgaxiomintmeasureequalintegral}

\item{$R_{a\vee b}(x) \stackrel{a.e}{=} R_a(x)\vee R_b(x)$ \ \ \ \ \ \ (for each $a,b \in B$),}\label{Stoneprobalgaxiomdisjunction}

\item{$R_{a'}(x) \stackrel{a.e}{=} 1-R_a(x)$ \ \ \ \ \ \ (for each $a \in B$).}\label{Stoneprobalgaxiomcomplement}

\end{enumerate}

Note that in axiom \ref{Stoneprobalgaxiomdisjunction}, the notation $\vee$ in the left side of the equality addresses the Boolean algebra operation while in the right side refers to the logical connective "max" between two formulas (as defined after Definition \ref{defformulasinintegrallogic}).
We will show that $T$ is finitely satisfiable.
Let $T_0$ be a finite subset of axioms of $T$.
Let $B_0$ be a finite sub measure algebra of $B$ containing every $a\in B$
for which $R_a$ appears in axioms in $T_0$.
Also let $M=\{a_1,...,a_k\}$ be the atomic elements of $B_0$, where $a \in B_0$ is called an atom of $B_0$ if given any $b \in B_0$ such that $b \leqslant a$, either $b=0$ or $b=a$.
Then, it is not hard to see that $\mu$ induces a probability measure $\nu$ on finite space $(M,P(M))$.
We prove the satisfiability of $T_0$ by making a model of it over the underlying finite measure space $\mathcal{M}=(M,P(M),\nu)$.
For that, we need to interpret relation symbols $R_a$'s ($a \in B$) in $\mathcal{M}$.
For each $a\in B_0$, interpret $R_a$ with the function $R_a^M$ defined by $R_a^M(a_i)=1$ if $a_i\leqslant a$ and $=0$ otherwise, for any $a_i \in M$.
Also for any $a \in B \setminus B_0$, interpret $R_a$ with any arbitrary $\{0,1\}$-valued function on $M$.
Then, it is not very difficult to see that the resulting $\mathcal L$-structure is a model of $T_0$.
This shows that $T$ is finitely satisfiable.

By logical compactness theorem (Theorem \ref{logicalcompactnesstheorem}), $T$ has a model, say $(M,\mathcal{C},\bar\mu; R^M_a)_{a\in B}$, where each $R^M_a$ is the interpretation of the relation symbol $R_a$ in this model.
Note that by definition of a model, $\bar\mu(M)=1$.
Also each $R^M_a$ is a measurable function on $M$ with respect to the $\sigma$-algebra $\mathcal{C}$.
Let $\mathcal{B} \subseteq \mathcal{C}$ be the smallest $\sigma$-algebra making every $R^M_a$ measurable. Also restrict $\bar\mu$ to $\mathcal{B}$ and still denote the restricted measure by $\bar\mu$. We claim that $(M,\mathcal{B},\bar\mu)$ is the desired measure space whose associated probability algebra is $\sigma$-order-continuous isomorphic to $B$.
It is not hard to see that by axiom \ref{StoneprobalgaxiomRa0or1}, each $R_a^M$ is a characteristic function (up to a null set).
Let $X_a:=\{x \in M: R_a^M(x)=1\}$ for every $a \in B$. Obviously, every $X_a$ belongs to $\mathcal{B}$. 
For every $A \in \mathcal{B}$, let $[A]$ to be the equivalence class of $A$ in $D$, where we define $D$ to be the associated probability algebra to $(M,\mathcal{B},\bar\mu)$.
Since each $R_a^M$ is a characteristic function (up to a null set) of the subset $X_a$, it is easy to see that the measure algebra $D$ is the same as the measure algebra associated to the restriction of the measure space $(M,\mathcal{B},\bar\mu)$ to the sub $\sigma$-algebra generated by $X_a$'s.

Define $\phi: B \rightarrow D$ by $\phi(a):=[X_a]$. We claim that $\phi$ is a measure algebra $\sigma$-order-continuous isomorphism.
We first check the injectivity of $\phi$.
Assume that $\phi(a)=\phi(b)$ for some $a,b \in B$. So $[X_a]=[X_b]$ which follows that $X_a \stackrel{a.e}{=} X_b$ with respect to the measure $\bar \mu$.
Then $X_a \triangle X_b$ is null. 
One can use the axioms to show that $X_a \triangle X_b \stackrel{a.e}{=} X_{a \underline\triangle b}$ where by $a \underline\triangle b$ in $B$ we mean the elements $(a \wedge b') \vee (a' \wedge b)$.
So $\bar \mu(X_{a \underline\triangle b})=0$.
Hence, by axiom \ref{Stoneprobalgaxiomintmeasureequalintegral}, we have
$\mu(a \underline\triangle b)=\int R_{a \underline\triangle b}^M=\bar\mu(X_{a \underline\triangle b})=0$.
Now, by definition of a measure algebra, we have $a \underline\triangle b=\textbf{0}$ which follows that $a=b$.
Therefore, $\phi$ is injective. It is also easy to see that $\phi(a')=\phi(a)'$ for every $a \in B$.

\vspace{1mm}

\textit{Claim.} Let $(b_i)_{i<\omega}$ be a sequence of elements of $B$. Then, $\phi(\bigvee_{i<\omega}b_i)=\bigvee_{i<\omega}\phi(b_i)$ and $\phi(\bigwedge_{i<\omega}b_i)=\bigwedge_{i<\omega}\phi(b_i)$.

\vspace{1mm}

\textit{Proof of Claim.}
First assume that $(b_i)_{i < \omega}$ is an increasing sequence of elements of $B$ and let $b:=\sup_{i<\omega}b_i$.
By using the axioms, it is easy to see that $X_{b_i} \stackrel{a.e}{\subseteq} X_b$ and $X_{b_i} \stackrel{a.e}{\subseteq}  X_{b_{i+1}}$ for each $i$.
So, $\bigcup_{i<\omega} X_{b_i} \stackrel{a.e}{\subseteq} X_b$. 
On the other hand, again by axioms, we have $\bar\mu(X_b)= \int R^M_b=\mu(b)$ and similarly, $\bar\mu(X_{b_i})=\mu(b_i)$ for each $i$.
Since $(b_i)_{i<\omega}$ is an increasing sequence in the measure algebra $B$, by a known fact (see for example in \cite{FremlinMeasuretheoryvol3}-321B), we have $\mu(\sup_{i<\omega}(b_i))=\lim_{i \rightarrow \infty} \mu(b_i)=\sup_{i<\omega} \mu(b_i)$. So we have

$$\bar\mu(\bigcup_{i<\omega} X_{b_i})=\sup_{i<\omega}\bar\mu(X_{b_i})
=\sup_{i<\omega}\mu(b_i)=\mu(\sup_{i<\omega}(b_i))=\mu(b)=\bar\mu(X_b).$$

Combination of the above facts follows that $X_b \stackrel{a.e}{=} \bigcup_{i < \omega} X_{b_i}$.
Thus, $[X_b]=[\bigcup_{i<\omega} X_{b_i}]$. 
Moreover, we have
$$\phi(\bigvee_{i<\omega}b_i)=\phi(b)=[X_b]=[\bigcup_{i < \omega} X_{b_i}]=\bigvee_{i<\omega}[X_{b_i}]=\bigvee_{i<\omega}\phi(b_i). \ \ \ \ \ \ (1)$$

Now assume that $(b_i)_{i<\omega}$ is an arbitrary (not necessarily increasing) sequence of elements of $B$ and let $b:=\sup_{i<\omega}b_i$. Let $c_i:=\bigvee_{j=1}^{i} b_j$. 
Now $(c_i)_{i<\omega}$ is an increasing sequence and by $(1)$, $\phi(\bigvee_{i<\omega}c_i)=\bigvee_{i<\omega}[X_{c_i}]$.
So 
$$\phi(\bigvee_{i<\omega}b_i)=\phi(\bigvee_{i<\omega}c_i)=\bigvee_{i<\omega}[X_{c_i}]=\bigvee_{i<\omega}[X_{\bigvee_{j=1}^{i} b_j}]=\bigvee_{i<\omega}(\bigvee_{j=1}^i[X_{b_j}])=\bigvee_{i<\omega} [X_{b_i}]=\bigvee_{i<\omega}\phi(b_i).$$

Moreover, by using this, we also have
$$\phi(\bigwedge_{i<\omega}b_i)=\phi((\bigvee_{i<\omega}b_i')')=
(\phi(\bigvee_{i<\omega}b_i'))'
=(\bigvee_{i<\omega}\phi(b_i'))'
=(\bigwedge_{i<\omega}(\phi(b_i'))')
=\bigwedge_{i<\omega}\phi(b_i).$$

It completes the proof of the claim.
\hfill \textit{Claim} $\square$
\end{proof}

Now we prove the surjectivity of $\phi$.
We remind from above that measure algebra $D$ is the same as the measure algebra associated to the measure space $(M,\mathcal{B}',\bar\mu|_{\mathcal{B}'})$, where $\mathcal{B}'$ is the sub $\sigma$-algebra generated by $X_a$'s.
But by definition of a generated $\sigma$-algebra, $\mathcal{B}'$ is 
the closure of the family of basic sets $X_a$'s under the operations "countable unions", "countable intersections" and "complement". 
So, it is not hard to verify that for showing that $\phi$ is surjective, it would be enough to prove that for any sequence $(b_i)_{i < \omega}$ of elements of $B$,
$\bigvee_{i<\omega} [X_{b_i}]$ and $\bigwedge_{i<\omega} [X_{b_i}]$ are in the image of $\phi$.
But by the above claim, we have 
$\bigvee_{i<\omega} [X_{b_i}]=\bigvee_{i<\omega}\phi(b_i)=\phi(\bigvee_{i<\omega}b_i) \in \phi(B)$ and 
$\bigwedge_{i<\omega} [X_{b_i}]=\bigwedge_{i<\omega}\phi(b_i)=\phi(\bigwedge_{i<\omega}b_i) \in \phi(B)$.
It follows that $\phi$ is surjective.
Similarly, using the above claim and arguments, it is not hard to see that $\phi$ is $\sigma$-order-continuous and measure-preserving Boolean isomorphism. So, it is a measure algebra $\sigma$-order-continuous isomorphism.
\hfill $\square$

\subsection{Daniell-Stone theorem for Daniell integrals}
In this subsection, we give a new proof for the classical Daniell-Stone theorem. Like the proofs of the other theorems in this paper, this proof is also using the logical compactness theorem as an essential tool.
One can find a classical proof for Daniell-Stone theorem for example in the book \cite{RoydenRealAnalysis}.
Let $\mathcal{A}$ be a vector lattice (always over $\mathbb{R}$ in this paper) of real functions on a set $X$
containing $1_X$ (the function with value $1$ on every $x \in X$). A Daniell integral on $\mathcal{A}$ is a positive
linear order-continuous real-valued function $I$ on $\mathcal{A}$ where
by order-continuity is meant $I(f_n)\downarrow0$ whenever $f_n\downarrow0_X$ pointwise where $0_X$ is the function with value $0$ on every $x \in X$. Since now on, when there is no danger of confusion, we use $1$ and $0$ instead of $1_X$ and $0_X$.
The Daniell-Stone theorem roughly states that there exists a measure $\mu$
on $X$ such that $I$ is the integration with respect to $\mu$.
By a lattice-linear combination in a vector lattice, we mean an expression obtained by combination of finitely many elements of the vector lattice by some linear and lattice operations. For example, if $x,y, z$ are elements of a vector lattice, then $x+(2y\wedge z)$ and $(x\vee y)\wedge(y-z)$ are some lattice-linear combinations.

\begin{thm}\label{DaniellStone}
\emph{(Daniell-Stone theorem)} 
Let $\mathcal{A}$ be a vector lattice of real functions on $X$
such that $1\in \mathcal{A}$. Let $\mathcal{E}$ be the $\sigma$-algebra generated by
$\mathcal{A}$ (i.e. the smallest $\sigma$-algebra making every function in $\mathcal{A}$ measurable).
Then, for each Daniell integral $I$ on $\mathcal A$, there is a measure $\rho$ on
$\mathcal{E}$ such that $I(f)=\int f d\rho$ for every $f\in \mathcal{A}$.
\end{thm}
\begin{proof} A standard argument shows that if $I(f)=\int fd\mu$ holds for every bounded function
$f \in \Ac$, then it holds for every function in $\Ac$ as well.
Also if $I(1)=0$, then it is not hard to see that $I(f)=0$ for every $f \in \mathcal{A}$ which gives rise to a measure $\mu$ with $\mu(X)=0$.
So, we may assume, without loss of generality, that $\Ac$ is a vector lattice of bounded functions and that $I(1)=1$.

Let $\mathcal L$ be a language (in integration logic) consisting of a constant symbol $c_a$ for each
$a\in X$ and a unary relation symbol $R_f$ for every $f \in \mathcal{A}$. Also we let universal bound of each relation symbol $R_f$ to be equal to an arbitrary upper bound of the function $|f|$.
In particular, for every $r \in \mathbb{R}$, $R_r$ is the relation symbol $R_f$ where $f$ is the constant function $r$.
Let $T$ be a $\mathcal{L}$-theory consisting of the following expressions (axioms) which can be written as some statements in integration logic in language $\mathcal{L}$ (one can get help from Remark \ref{somesampleexpresibleproperties} for stating them).

\begin{enumerate}
\item{$\e(c_a,c_b)=0$ \ \ \ \ \ \ (for every distinct $a,b\in X$),} \label{Daniellaxiomconstantsseparated}
\item{$R_f(c_a)=f(a)$ \ \ \ \ \ \ (for each $a \in X$ and $f \in \mathcal{A}$),}\label{Daniellaxiomforconstants}
\item{$R_r(x) \stackrel{a.e}{=} r$ \ \ \ \ \ \ (for each $r \in \mathbb{R}$),}\label{DaniellaxiomforconstantR_r_is_r}
\item{$R_{f + g} \stackrel{a.e}{=} R_f + R_g$  \ \ \ \ \ \ (for every $f,g \in \mathcal{A}$),}\label{DaniellStoneaxiomRf+g=Rf+Rg}
\item{$R_{rf} \stackrel{a.e}{=} rR_f$  \ \ \ \ \ \ (for each $f\in \mathcal{A}$ and $r\in \mathbb{R}$),}\label{DaniellStoneaxiomRrf=rRf}
\item{$R_{f \vee g} \stackrel{a.e}{=} R_f \vee R_g$ \ \ \ \ \ \ (for every $f,g \in\mathcal{A}$),}\label{Daniellaxiomforunion}
\item{$\int R_f=I(f)$ \ \ \ \ \ \ (for each $f \in \mathcal{A}$).}\label{DaniellStoneaxiomintRf=If}
\end{enumerate}

Note that in axioms
\ref{DaniellStoneaxiomRf+g=Rf+Rg}, \ref{DaniellStoneaxiomRrf=rRf} and \ref{Daniellaxiomforunion},
the notations $+$, scalar multiplication $rf$ and $\vee$ in the lefts sides of the equalities refer to the corresponding operations in the vector lattice $\mathcal{A}$ while in the right sides address to the logical connective "+", "." and "max" between formulas as defined in Definition \ref{defformulasinintegrallogic} and after it.

We first want to prove that $T$ is finitely satisfiable. 
For that, we instead start to show that $T$ is finitely approximately satisfiable (as defined before Lemma \ref{approximate}),
which by Lemma \ref{approximate} amounts to saying that $T$ is finitely satisfiable.
Let $T_0$ be a finite subset of axioms of $T$ and $f_1,f_2,...,f_t$ be the list
of functions in $\mathcal{A}$ for which $R_{f_i}$'s appear in axiom \ref{DaniellStoneaxiomintRf=If} in $T_0$.
We show that $T_0$ is approximately satisfiable.
Without loss of generality, we may assume that the function $1_X$ is among $f_1,\ldots,f_t$ (otherwise, we can add the instance of axiom \ref{DaniellStoneaxiomintRf=If} for the function $1_X$ to the current list of our axioms in $T_0$ and prove approximate satisfiability of this larger set of axioms, which of course follows the approximate satisfiability of $T_0$).
So, without loss, we assume that $f_1=1_X$.
Fix $0<\epsilon<\frac{1}{2}$.
We will try to build a $\mathcal{L}$-structure on the domain set $X$ which $\epsilon$-approximately (with error at most $\epsilon$) satisfies $T_0$
(for example, if the axiom $"\int R_{f_2}=I(f_2)"$, the instance of axiom \ref{DaniellStoneaxiomintRf=If} for $f_2$, belongs to $T_0$, then we would have to show that $\int R_{f_2} \leqslant I(f_2) \pm \epsilon$ holds in the structure).
In order to build such a $\mathcal{L}$-structure, it would be enough to interpret every relation symbols $R_{f}$ and constant symbols $c_a$ in $X$ and also define a measure on $X$ in such a way that the axioms of $T_0$ hold with error at most $\epsilon$.
For that, interpret each relation symbol $R_{f} \in \mathcal{L}$ 
 by the function $f$ on $X$ itself. 
Moreover, for each $a \in X$, interpret the constant symbol $c_a$ in $X$ 
by the element $a$ itself.
It is easy to see that the equality of any instance of the
axioms \ref{Daniellaxiomconstantsseparated}, \ref{Daniellaxiomforconstants}, \ref{DaniellaxiomforconstantR_r_is_r}, \ref{DaniellStoneaxiomRf+g=Rf+Rg}, \ref{DaniellStoneaxiomRrf=rRf} and \ref{Daniellaxiomforunion}
appearing in $T_0$ holds (in exact way, which is even stronger than a.e) 
in this structure.
For axiom \ref{DaniellStoneaxiomintRf=If}, we need
a probability measure $\lambda$ on $X$
constructed in a way that the instances of axiom \ref{DaniellStoneaxiomintRf=If}, where $R_{f_i}$'s in them are interpreted by $f_i$'s, are satisfied by error at most $\epsilon$ (or equivalently $|\int f_i \ d\lambda-I(f_i)| \leqslant \epsilon$ for each $1 \leqslant i \leqslant t$).
Note that because of linearity of integral and $I$, obviously every such $\lambda$ on $X$ which satisfies inequalities $|\int f_i \ d\lambda-I(f_i)| \leqslant \epsilon$, also satisfies inequalities $|\int f'_i \ d\lambda-I(f'_i)| \leqslant \epsilon$ and vice versa, where $f'_i:=f_i+c$ for each $i$ and $c$ is a big enough positive real number such that each $f'_i$ is a positive-valued function (recall that $f_i$'s are bounded functions).
So we may assume from beginning that $f_1,\ldots,f_t$ are positive-valued functions.
We construct the measure $\lambda$ as follows .

Let $J:=[0,\alpha)$ contains the range of every $f_i$ ($1 \leqslant i \leqslant t$).
Use Lemma \ref{nullpoints} to find a partitioning $[u_1,u_2), [u_2,u_3),\ldots,[u_{s-1},u_s)$ of $J$, with $u_1=0$ and $u_s=\alpha$, such that each interval piece has length less that $\epsilon$ and each $u_j$ is an inessential value of each $f_i$ with respect to $I$ (as defined in Remark \ref{DecSeqCores}). 
Note that $u_1$ and $u_s$ are automatically inessential values of every $f_i$ since they are not in range of them.
Denote each interval $[u_j,u_{j+1})$ by $J_j$. Also denote the open interval $(u_j,u_{j+1})$ by $J^o_j$.
Let $\Bc_0$ be the Boolean algebra on $X$ generated by the family of subsets 
$f_i^{-1}(J_j)$ of $X$.
Also let $\mathcal P:=\{P_1,...,P_\ell\}$ be the family of atoms of the Boolean algebra $\Bc_0$. 
Clearly $\mathcal P$ is a partitioning for $X$.
Also it is not hard to see that for each $P_k$ we have
$P_k=\bigcap_{i=1}^{t}f_i^{-1}(J_{j_{k,i}})$ where $j_{k,i} \in \{1,\ldots,s-1\}$ for each $i$.
For each such $P_k$, we define $P^*_k:=\bigcap_{i=1}^{t}f_i^{-1}(J^o_{j_{k,i}})$.
Obviously, $P^*_k \subseteq P_k$ for each $k$.
Lemma \ref{tendtochar} gives us for each $P^*_k$ a particular sequence $(\xi^k_{n})_{n < \omega}$ of $[0,1]$-valued functions in $\Ac$ increasing pointwise to
$\chi(P^*_k)$ in such a way that the support of each function $\xi^k_{n}$ is a subset of $P^*_k$.
For each $P_k$, set $\lambda_0(P_k):=\lim_{n \rightarrow \infty} I(\xi^k_{n})$. Since $P_k$'s are the atoms of the Boolean algebra $\Bc_0$, $\lambda_0$ extends in the natural way to a measure (still denoted by $\lambda_0$) on the $\sigma$-algebra generated by $\Bc_0$, which is the same as $\Bc_0$ since $\Bc_0$ is finite.
We normalize the measure $\lambda_0$ and turn it to a probability measure $\lambda$ on $X$. Note that as we will see later after Claim 1, we have $1-\epsilon \leqslant \lambda_0(X)$ which follows that $\lambda_0(X) \not=0$. So normalization makes sense.
Now it is easy to see that, by using this measure $\lambda$, in fact we have obtained a $\mathcal{L}$-structure on the domain $(X,\Bc_0,\lambda)$.

Now it's time to verify that axiom \ref{DaniellStoneaxiomintRf=If} holds $\epsilon$-approximately in this obtained $\mathcal{L}$-structure. 
We remind that verifying this, completes the proof of $\epsilon$-approximately satisfiability of $T_0$.
For that, we must show that for each $i=1,\ldots,t$,
$|I(f_i) - \int f_i\ d\lambda| \leqslant \epsilon$.
It is easy to see that by the way we have defined $P_k$'s in above, for each $P_k$ and every $1 \leqslant i \leqslant t$ and every $x,y \in P_k$, we have $|f_i(x)-f_i(y)|< \epsilon$ (since $f_i(x)$ and $f_i(y)$ both belong to $J_{j_{k,i}}$).
Therefore, for every $i=2,\ldots,t$, we can find some nonnegative $\mathcal P$-simple function $h_i$ on $X$ (i.e. a function which has constant values on each $P_k$ but possibly different values on different $P_k$'s), not necessarily in $\mathcal{A}$, such that $0 \leqslant f_i(x)-h_i(x) \leqslant \epsilon$ for every $x \in X$.
In the particular case $i=1$, we define $h_1$ to be specifically the constant function $1_X$, which is clearly a $\mathcal P$-simple function and also satisfies $0 \leqslant f_1-1_X \leqslant \epsilon$ for every $x \in X$ (since $f_1=1_X$ as defined above).
Let $r_{i,k}$ be the constant value of $h_i$ on $P_k$. Hence, for each $1 \leqslant i \leqslant t$ we have $h_i= \sum_{k=1}^{\ell}r_{i,k}.\chi(P_k)$ where $r_{i,k}$'s are non-negative.
We have
$$|I(f_i) - \int f_i\ d\lambda| \leqslant |I(f_i) - \int h_i\ d\lambda|+ |\int f_i\ d\lambda-\int h_i\ d\lambda| = 
|I(f_i) - \int h_i\ d\lambda| + |\int (f_i-h_i)\ d\lambda| \leqslant |I(f_i) - \int h_i\ d\lambda|+\epsilon.$$
Therefore, in order to verify that axiom \ref{DaniellStoneaxiomintRf=If} holds $\epsilon$-approximately, it is enough to show that for each $i=1,\ldots,t$, $|I(f_i) - \int h_i\ d\lambda| \leqslant \epsilon$ (which, in turn follows that $|I(f_i) - \int f_i\ d\lambda| \leqslant 2\epsilon$ and then, by a suitable arrangement of $\epsilon$ in the beginning and replacing it by $\frac{\epsilon}{2}$, we get $|I(f_i) - \int f_i\ d\lambda| \leqslant \epsilon$ as desired).
So, we start to show that $|I(f_i) - \int h_i\ d\lambda| \leqslant \epsilon$ for each $i$.

We remind from above that for each $h_i$ we have $h_i= \sum_{k=1}^{\ell}r_{i,k}.\chi(P_k)$.
For each $1 \leqslant i\leqslant t$ and $n<\omega$ define $h_{i,n}:= \sum_{k=1}^{\ell}r_{i,k}.\xi^k_{n}$.
Since each $(\xi^k_{n})_{n < \omega}$ is (as defined above) an increasing sequence of $[0,1]$-valued functions in $\mathcal{A}$ converging to $\chi(P^*_k)$ with supports inside $P^*_k$ (where $P^*_k \subseteq P_k$) and $r_{i,k}$'s are non-negative, for every $i$, the sequence $(h_{i,n})_{n<\omega}$ is an increasing sequence of functions in $\mathcal{A}$ and for every point $x$ in each $P^*_k$, we have
$h_i(x)=\lim_{n \rightarrow \infty} h_{i,n}(x)$. Also for every $x$ outside of all $P^*_k$'s, we have $\lim_{n \rightarrow \infty} h_{i,n}(x)=0$.
Recall from above that $0 \leqslant f_i(x)-h_i(x) \leqslant \epsilon$ for every $x \in X$.
Now it is not difficult to see that for each $i$, the sequence $(f_i - h_{i,n})_{n<\omega}$ is a decreasing sequence of nonnegative functions.
For each 
$i=1,\ldots,t$, let $R_{i,n}:=(f_i - h_{i,n})\vee\epsilon$. So $(R_{i,n})_{n<\omega}$ is a decreasing sequence of nonnegative functions too.
Also for each $k$ and $x \in P^*_k$, the sequence $R_{i,n}(x)$ decreases to $\epsilon$ as $n$ tends to infinity.
It is easy to see that each $R_{i,n}$ is a function in $\mathcal{A}$.

\vspace{1mm}

\textit{Claim 1.} Let $i_0 \in \{1,\ldots,t\}$ be arbitrary. 
Then $\lim_{n \rightarrow \infty}I(R_{i_0,n}) = \epsilon$.

\vspace{1mm}

\textit{Proof of Claim 1.}
Fix an arbitrary small $\delta >0$. Let $H=\bigcup_{i=1}^{t}\bigcup_{j=1}^{s}f_i^{-1}(\{u_j\})$.
Also let $(\psi_n^{i,j})_{n<\omega}$ 
be the decreasing sequence corresponding to $u_j$ for $f_i$ (as defined in Remark \ref{DecSeqCores})
converging to $\chi(f_i^{-1}(\{u_j\}))$. Thus, each $\psi_n^{i,j}$ is a $[0,1]$-valued function in $\mathcal{A}$.
It is not hard to see that the sequence $(v_n)_{n<\omega}$ defined by $v_n:=\max_{i,j}\psi_n^{i,j}$ is a sequence of $[0,1]$-valued functions decreasing to $\chi(H)$ pointwise.
So, it is easy to see that for every $n < \omega$ and $x \in H$, $v_n(x)=1$.
Since every $u_j$ is an inessential value of each $f_i$ with respect to $I$ (see in above the way that $u_j$'s were defined), then for every $i,j$, we have $I(\psi_n^{i,j}) \downarrow 0$ as $n$ tends to infinity.
So, it is not difficult to verify that $I(v_n) \downarrow 0$ as $n$ tends to infinity.
Hence, by replacing the sequence $(v_n)_{n<\omega}$ with a suitable subsequence of it, we may assume, without loss, that $I(v_n)\leqslant \frac{\delta}{4^n}$ for each $n$.
For every $n <\omega$, define the function $g_n \in \mathcal{A}$ by $g_n:=\max_{m \leqslant n} m.v_m$.
Then, it is easy to see that $(g_n)_{n<\omega}$ is an increasing sequence of nonnegative functions with $g_n(x)=n$ at each $x \in H$. It makes $(g_n)_{n<\omega}$ increasing to $\infty$ at each $x \in H$. Since $v_m$'s are $[0,1]$-valued functions,
$g_n \leqslant \sum_{m \leqslant n} m.v_m$ for each $n$.
Therefore, for every $n$ we have
$$I(g_n) \leqslant I(\sum_{m \leqslant n} m.v_m)= \sum_{m \leqslant n} I(m.v_m) \leqslant \sum_{m \leqslant n} m.\frac{\delta}{4^m} \leqslant \delta.$$

Now we want to show that 
$(R_{i_0,n}-g_n) \vee \epsilon \downarrow \epsilon$ as $n$ tends to infinity at every $x \in X$.
Note that for each $n$, $(R_{i_0,n}-g_n) \vee \epsilon$ is a function in $\mathcal{A}$.
Since $(R_{i_0,n})_{n<\omega}$ is a decreasing sequence and $(g_n)_{n<\omega}$ increases to $\infty$ at each $x \in H$, we have $(R_{i_0,n}-g_n) \downarrow -\infty$ on $H$ as $n$ tends to infinity.
It is easily seen that for each $k=1,\ldots,\ell$, we have $P_k \setminus P^*_k \subseteq H$. 
So $(R_{i_0,n}-g_n) \vee \epsilon \downarrow \epsilon$ as $n$ tends to infinity at every $x \in \bigcup_{k=1}^{\ell} (P_k \setminus P^*_k)$.
On the other hand, we remind from above that for each $k$ and $x \in P^*_k$, the sequence $R_{i_0,n}(x)$ decreases to $\epsilon$ as $n$ tends to infinity.
Therefore, since $(g_n)_{n<\omega}$ is an increasing sequence of nonnegative functions, $(R_{i_0,n}-g_n) \vee \epsilon \downarrow \epsilon$ as $n$ tends to infinity at every $x \in \bigcup_{k=1}^{\ell} P^*_k$.
Combining the above facts, we have $(R_{i_0,n}-g_n) \vee \epsilon \downarrow \epsilon$ as $n$ tends to infinity at every $x \in  (\bigcup_{k=1}^{\ell} (P_k \setminus P^*_k)) \bigcup (\bigcup_{k=1}^{\ell} P^*_k)=X$.
It follows, by order-continuity of $I$, that $I((R_{i_0,n}-g_n) \vee \epsilon) \downarrow \epsilon$.
Hence, there exists $N_{\delta} \in \mathbb{N}$ such that for each
$n > N_{\delta}$, $I((R_{i_0,n}-g_n) \vee \epsilon)\leqslant \epsilon+\delta$.
Thus, for each $n > N_{\delta}$, $I(R_{i_0,n})-I(g_n)=I(R_{i_0,n}-g_n)\leqslant I((R_{i_0,n}-g_n) \vee \epsilon) \leqslant \epsilon+\delta$.
So, since $I(g_n) \leqslant \delta$ (as proved above), we have
$I(R_{i_0,n}) \leqslant \epsilon+2\delta$ for every $n > N_{\delta}$.
Since $\delta$ was chosen arbitrarily, we have $\lim_{n \rightarrow \infty}I(R_{i_0,n}) = \epsilon$.
\hfill \textit{Claim 1} $\square$

\vspace{1.5mm}

For each $i=1,\ldots,t$, since $f_i - h_{i,n}$ is a nonnegative function (for each $n$), we have
$$|I(f_i) - \int h_i\ d\lambda_0|=|I(f_i)-\sum_{k=1}^{\ell}r_{i,k}.\lambda_0(P_k)|
=|I(f_i)-\sum_{k=1}^{\ell}r_{i,k}.\lim_n I(\xi^k_{n})|
=|\lim_n I(f_i-h_{i,n})|
\leqslant \lim_n I(R_{i,n}) = \epsilon,$$

where we used Claim 1 in the last equality and also our definition of $\lambda_0(P_k)$ defined as $\lim_{n \rightarrow \infty} I(\xi^k_{n})$ in the second equality.
Specifying the above inequality for $f_1$ and $h_1$, we get $|I(f_1) - \int h_1\ d\lambda_0| \leqslant \epsilon$ where we remind that we had assumed $f_1=1_X$ and $h_1=1_X$. 
So, since $I(1_X)=1$ and $\int h_1\ d\lambda_0=\int d\lambda_0=\lambda_0(X)$, we have $|\lambda_0(X)-1| \leqslant \epsilon$.
Thus, $1-\epsilon \leqslant \lambda_0(X) \leqslant 1+\epsilon$.
Since $\lambda$ is the normalization of $\lambda_0$, we have $\lambda=\frac{1}{\lambda_0(X)}\lambda_0$. So
$\frac{1}{1+\epsilon} \lambda_0 \leqslant \lambda \leqslant \frac{1}{1-\epsilon} \lambda_0$.
Now, using above inequalities, for each $1\leqslant i \leqslant t$ we have
$$\big|I(f_i) - \int h_i\ d\lambda\big| \leqslant \big|I(f_i) - \int h_i\ d \lambda_0\big|+\big|\int h_i\ d\lambda_0 - \int h_i\ d\lambda\big| \leqslant \epsilon +\big|\int h_i\ d\lambda_0 - \frac{1}{\lambda_0(X)}\int h_i\ d\lambda_0\big|$$
$$= \epsilon + \big|(1-\frac{1}{\lambda_0(X)})\big| \ \big| \int h_i\ d\lambda_0\big| \leqslant \epsilon + \big|(\frac{\lambda_0(X)-1}{\lambda_0(X)})\big| \sup h_i. \lambda_0(X)=\epsilon +\big| \lambda_0(X)-1 \big| \sup h_i \leqslant \epsilon + \epsilon\sup h_i$$
$$\leqslant \epsilon(1 + \sup f_i).$$ 
So, by a suitable arrangement for $\epsilon$ from the beginning, one guarantees the axiom \ref{DaniellStoneaxiomintRf=If} to be also approximately satisfied by error at most $\epsilon$ in the constructed $\mathcal{L}$-structure.
It follows that $T_0$ is approximately satisfiable with error at most $\epsilon$.
Consequently, since $\epsilon$ was arbitrary, $T_0$ is approximately satisfiable.
It follows that $T$ is finitely-approximately satisfiable and hence, by Lemma \ref{approximate}, finitely satisfiable.
It finishes the step of proving the finitely satisfiability of $T$.

\vspace{1.5mm}

Now, in the next step of the proof, by logical compactness theorem (Theorem \ref{logicalcompactnesstheorem}), one concludes that $T$ has a model, say $\mathcal{N}$.
Let $(N,\Bc_1,\nu_1)$ be the underlying probability space of the model $\mathcal{N}$.
Define $\mathcal{K}_0:=\{R_f^N:f \in \mathcal{A}\}$ and let $\mathcal{K}$ to be the $\mathbb{R}$-vector lattice of functions generated by $\mathcal{K}_0$ and constant functions
(so, every constant real function and every lattice-linear combination, for example $(R_{f_1}^N \vee R_{f_2}^N)+R_{f_3}^N$, belongs to $\mathcal{K}$). 
By definition of a model, interpretation of every formula 
is a measurable function with respects to the $\sigma$-algebra $\Bc_1$. 
Note that every function in $\mathcal{K}$ is the interpretation of some formula.
So, every function in $\mathcal{K}$ is measurable.
Let $\Bc \subseteq \Bc_1$ be the minimal $\sigma$-algebra on $N$ making every function in $\mathcal{K}$ 
measurable. Also let $\nu$ be the restriction of $\nu_1$ to $\Bc$.
We consider the measure space $(N,\Bc,\nu)$.
By axioms
\ref{Daniellaxiomconstantsseparated} and \ref{Daniellaxiomforconstants} of $T$ and the fact that our model satisfies them, it is not hard to see that we may assume, without loss, that $X \subseteq N$ (by identifying every $a \in X$ with the interpretation of the constant symbol $c_a$ in $N$) and that each $h\in \mathcal{A}$ is the restriction of $R^N_h$ to $X$.
It easily follows that if we take any member of $\mathcal{K}$, say $\theta:=\sigma(R_{h_1}^N,\ldots,R_{h_m}^N)$ for some lattice-linear combination $\sigma$ of $R_{h_1}^N,\ldots,R_{h_m}^N$ for some $h_1,\ldots,h_m \in \mathcal{A}$, then the restriction $\theta|_X$ is exactly the function 
$\sigma(h_1,\ldots,h_m)$ which is a function in $\mathcal{A}$.
Moreover, by using the axioms, it is easy to see that $R_{\sigma(h_1,\ldots,h_m)}^N \stackrel{a.e}{=} \theta$.
In other words, for every $\theta \in \mathcal{K}$, we have $R_{\theta|_X}^N \stackrel{a.e}{=} \theta$.

\vspace{1mm}

Let $\mu$ be the subspace measure on $X$ induced by $\nu$. We remind that the construction of subspace measures was briefly reviewed in 
Subsection \ref{prelimmeasureintlog}.

\vspace{1mm}

\textit{Claim 2.} We have $\mu(X) = 1$, which amounts to saying that $X$ has full outer measure in $N$ with respect to the measure $\nu$.

\vspace{1mm}

\textit{Proof of Claim 2.}
Let $\mathcal{C}$ be the Boolean algebra
generated by the sets $\theta^{-1}(r,\infty)$ in $N$ where $\theta \in \mathcal{K}$ and $r \in \mathbb{R}$.
Note that for every $\theta \in \mathcal{K}$ and $r \in \mathbb{R}$, we have $\theta^{-1}(r,\infty)=\theta'^{-1}(0,\infty)$ where $\theta'=\theta-r$, and since $\mathcal{K}$ is a vector lattice, $\theta' \in \mathcal{K}$.
So $\mathcal{C}$ is the Boolean algebra generated by the sets $\theta^{-1}(0,\infty)$ in $N$ where $\theta \in \mathcal{K}$.
By the minimality of $\Bc$ mentioned above, it is easily seen that $\Bc$ is the $\sigma$-algebra generated by $\mathcal{C}$.
So, since $\nu$ is $\sigma$-finite, by a usual extension theorem in measure theory (see for example Theorem A p.54 of \cite{HalmosMeasuretheory}), there exists a unique
extension of $\nu |_{\mathcal{C}}$ to $\mathcal{B}$ and it is $\nu$ itself.
But, on the other hand, Carath\'eodory extension theorem (Theorem \ref{caratheodoryextensiontheorem}) extends $\nu |_{\mathcal{C}}$ to $\mathcal{B}$.
It follows that $\nu$ on $\mathcal{B}$ is the same as the measure obtained by the Carath\'eodory extension process from $\nu |_{\mathcal{C}}$.
Hence, by Carath\'eodory extension process explained in the beginning of the paper, for each $U \in \mathcal{B}$ we have
$$\nu(U)=\inf \Big\{\sum_{i < \omega} \nu(U_i): \ \ U \subseteq \bigcup_{i < \omega} U_i, U_i \in \mathcal{C}\Big\}.$$

Now by definition of subspace measure and above facts, we have
$$\mu(X)=\inf \Big\{ \nu(U): X \subseteq U \in \mathcal{B}\Big\}=\inf \Big\{ \inf \Big\{\sum_{i < \omega} \nu(U_i): \ \ U \subseteq \bigcup_{i < \omega} U_i, U_i \in \mathcal{C}\Big\}: X \subseteq U \in \mathcal{B} \Big\}$$
$$=\inf \Big\{ \sum_{i < \omega} \nu(U_i): \exists U \ s.t \ X \subseteq U \in \mathcal{B}, \ \ U \subseteq \bigcup_{i < \omega} U_i, U_i \in \mathcal{C} \Big\}=\inf \Big\{ \sum_{i < \omega} \nu(U_i): X \subseteq \bigcup_{i < \omega} U_i, U_i \in \mathcal{C} \Big\}.$$

Let $\{U_i\}_{i<\omega} \subseteq \Cc$ be a covering of $X$.
To complete the proof of Claim 2, it is enough to show that $1 \leqslant \sum_{i<\omega}\nu(U_i)$.
As mentioned above, $\mathcal{K}$ is a $\mathbb{R}$-vector lattice of real measurable functions on $(N,\mathcal{B},\nu)$ and $\mathcal{C}$ is the Boolean algebra generated by the sets $\theta^{-1}(0,\infty)$ in $N$ where $\theta \in \mathcal{K}$.
So, by applying Lemma \ref{changecover}, 
for every $\epsilon>0$, one can find a countable covering of $X$ of subsets of $N$ of the form $\theta^{-1}(0,\infty)$ with $\theta \in \mathcal{K}$ with the property $\nu(\theta^{-1}(\{0\}))=0$ in such a way that sum of their $\nu$-measures does not differ from sum of $\nu$-measures of $U_i$'s with more than $\epsilon$.
Thus, if we manage to prove that for each $\epsilon >0$, sum of the $\nu$-measures of members of such mentioned covering corresponding to $\epsilon$ obtained by Lemma \ref{changecover} is at least $1$, then we conclude that $1 \leqslant \sum_{i<\omega}\nu(U_i)$ as desired and we would be done. 
So, by abuse of notations, we may assume from the beginning that $\{U_i\}_{i<\omega}$ is such a covering and for each $i$, $U_i=\theta_i^{-1}(0,\infty)$
for some $\theta_i \in \mathcal{K}$ and moreover, $\nu(\theta_i^{-1}(\{0\}))=0$.
So now, we only need to show that in this particular covering of $X$ with the mentioned properties, $1 \leqslant \sum_{i<\omega}\nu(U_i)$ holds.

Define $V_i:=U_i \cap X$ for each $i$. Then, clearly $\{V_i\}_{i<\omega}$ is a covering for $X$.
Also, for each $i <\omega$, we have $V_i=f_i^{-1}(0,\infty)$ where $f_i:=\theta_i|_X$. 
We remind that we are viewing (by using axioms \ref{Daniellaxiomconstantsseparated} and \ref{Daniellaxiomforconstants}) $X$ as a subset of $N$ and moreover, 
as mentioned above, restriction of any member of $\mathcal{K}$ to $X$ belongs to $\mathcal{A}$.
So we have $f_i \in \mathcal{A}$ for each $i<\omega$.
Also, as we had mentioned earlier, we have $R_{f_i}^{N} \stackrel{a.e}{=} \theta_i$.
By Lemma \ref{tendtochar}, for each $i<\omega$, there is a particular increasing sequence $(f_{i,n})_{n<\omega}$
of $[0,1]$-valued functions in $\mathcal{A}$ converging to $\chi(V_i)$ pointwise
and that the support of each function in the sequence is a subset of $V_i$.

\vspace{1mm}

\textit{Subclaim.} For each $i,n < \omega$, we have $\int R_{f_{i,n}}^{N}\ d\nu \leqslant \nu(U_i)$.

\vspace{1mm}

\textit{Proof of Subclaim.}
We first define some notions.
We call a pair $(f,g)$ of real-valued functions on a domain set a special pair if for all $x$ in the domain, firstly, $0 \leqslant f(x) \leqslant 1$, and secondly, if $g(x) <0$ then $f(x)=0$.
If the domain is a measure space, we call a pair $(f,g)$ almost special if the same conditions hold 
when we replace "for all $x$" with "for almost all $x$ with respect to the measure on the domain".
Also for every two functions $f$ and $g$ 
over a domain, we use the notation $f * g$ for denoting the function
$(f \vee (-g \vee 0))-f-(-g \vee 0)$. 
It is not difficult to see that a pair $(f,g)$ of functions on a domain is a special pair if and only if
$f*g=0$ and $0 \leqslant f \leqslant 1$ at every point.
Similarly, it is not hard to verify that a pair $(f,g)$ of functions on a measure space is an almost special pair if and only if
$\int |f*g|=0$, $\int (f \wedge 0)=0$ and $\int ((f \vee 1)-1)=0$.
Note that by using the axioms of theory $T$, it is not hard to see that for every $f,g \in \mathcal{A}$, we have $R_{f}*R_{g} \stackrel{a.e}{=} R_{f*g}$.
Also we remind that by axiom \ref{DaniellaxiomforconstantR_r_is_r}, for every constant function $r$, we have $R_r \stackrel{a.e}{=} r$.

Now we start to show that for any $i$ and $n$, we have $\int R_{f_{i,n}}^{N}\ d\nu \leqslant \nu(U_i)$.
Fix any arbitrary $i$ and $n$.
Since 
$f_{i,n}$ is $[0,1]$-valued and takes value $0$ outside $V_i$ and also $V_i=f_i^{-1}(0,\infty)$ (see above), the pair $(f_{i,n},f_i)$ is a special pair (on the domain $X$).
So, $f_{i,n}*f_i=0$ and $0 \leqslant f_{i,n} \leqslant 1$ at every point.
We claim that $(R_{f_{i,n}}^{N},R_{f_i}^{N})$ is an almost special pair on $(N,\mathcal{B},\nu)$.
In order to show this, we verify the equivalent condition to being almost special mentioned in the previous paragraph.
By using the axioms and above facts, we have
$$\int |R_{f_{i,n}}^{N}*R_{f_i}^{N}| \ d \nu=\int |R_{f_{i,n}*f_i}^{N}| \ d\nu=
\int ((R_{f_{i,n}*f_i}^{N} \vee 0)-(R_{f_{i,n}*f_i}^{N} \wedge 0)) \ d\nu
=\int ((R_{f_{i,n}*f_i}^{N} \vee R_0^N) \ d\nu - \int (R_{f_{i,n}*f_i}^{N} \wedge R_0^N)) \ d\nu$$
$$=\int R_{((f_{i,n}*f_i)\vee 0)}^{N} \ d\nu - \int R_{((f_{i,n}*f_i)\wedge 0)}^{N} \ d\nu=
I((f_{i,n}*f_i) \vee 0)-I((f_{i,n}*f_i) \wedge 0)=I(0)-I(0)=0.$$
Also we have
$\int (R_{f_{i,n}}^{N} \wedge 0) \ d\nu=\int (R_{f_{i,n}}^{N} \wedge R_0^N)=\int (R_{f_{i,n}\wedge 0}^{N})=I(f_{i,n}\wedge 0)=I(0)=0$.
Similarly, 
$$\int ((R_{f_{i,n}}^{N} \vee 1)-1) \ d\nu= \int ((R_{f_{i,n}}^{N} \vee R_1^N)-R_1^N) \ d\nu =
\int R_{(f_{i,n} \vee 1)-1}^{N} \ d\nu =I((f_{i,n} \vee 1)-1)=I(1-1)=0.$$
Therefore, $(R_{f_{i,n}}^{N},R_{f_i}^{N})$ is an almost special pair on $(N,\mathcal{B},\nu)$.
We remind from above that $R_{f_i}^{N} \stackrel{a.e}{=} \theta_i$. 
Thus, it is easily seen that $(R_{f_{i,n}}^{N},\theta_i)$ is an almost special pair on $(N,\mathcal{B},\nu)$.
Hence, for $\nu$-almost all $x \in N$, if the function $\theta_i$ has negative value on $x$, then the function $R^{N}_{f_{i,n}}$ takes value $0$ on that $x$.
It follows that $\int_{U'_i} R_{f_{i,n}}^{N}\ d\nu=0$ where $U'_i:=\theta_i^{-1}(-\infty,0)$.
Furthermore, as mentioned before, we have $\nu(U''_i)=0$ where 
$U''_i:=\theta_i^{-1}(\{0\})$. 
We remind that $U_i=\theta_i^{-1}(0,\infty)$.
So, we have
$$\int_N R_{f_{i,n}}^{N}\ d \nu = \int_{U_i} R_{f_{i,n}}^{N}\ d\nu+\int_{U'_i} R_{f_{i,n}}^{N}\ d\nu+\int_{U''_i} R_{f_{i,n}}^{N}\ d\nu=\int_{U_i} R_{f_{i,n}}^{N}\ d\nu
\leqslant \nu(U_i),$$
where in the last inequality, we used the fact that $R_{f_{i,n}}^{N}$ is $[0,1]$-valued almost everywhere (as mentioned above).
It completes the proof of the subclaim.
\hfill \textit{Subclaim} $\square$

\vspace{1.5mm}

For every $n < \omega$, define $g_n:=f_{1,n}\vee\ldots\vee f_{n,n}$.
Obviously, every $g_n$ belongs to $\mathcal{A}$.
Since $V_i$'s cover $X$ and for each $i$, the sequence $(f_{i,n})_{n<\omega}$ increases to $\chi(V_i)$,
it is not hard to see that the sequence $(g_n)_{n<\omega}$ increases pointwise to $1_X$. Thus, by order-continuity of $I$, $\lim_{n \rightarrow \infty} I(g_n)=1$.
We remind that in the proof of the above subclaim, it was proven that for each $i$ and $n$, $(R_{f_{i,n}}^{N},R_{f_i}^{N})$ is an almost special pair on $(N,\mathcal{B},\nu)$. So, in particular,
for each $i$ and $n$, we have $0  \stackrel{a.e}{\leqslant} R_{f_{i,n}}^{N}  \stackrel{a.e}{\leqslant} 1$. 
Hence, by using axiom \ref{Daniellaxiomforunion}, we have
$R_{g_n}^{N} \stackrel{a.e}{=} \bigvee_{i=1}^n R_{f_{i,n}}^{N} \stackrel{a.e}{\leqslant} \sum_{i=1}^n R_{f_{i,n}}^{N}$.
Therefore, by using axiom \ref{DaniellStoneaxiomintRf=If}, we have
$$I(g_n)=\int R_{g_n}^{N}\ d\nu \leqslant\sum_{i=1}^n\int R_{f_{i,n}}^{N}\ d\nu.\ \ \ \ \ \ (*)$$

By combining  $(*)$ and the fact that for each $i$ and $n$ we have $\int R_{f_{i,n}}^{N}\ d\nu \leqslant \nu(U_i)$ (the above subclaim), we get

$$1=\lim_{n \rightarrow \infty}I(g_n) \leqslant \lim_{n \rightarrow \infty}(\sum_{i=1}^n\int R_{f_{i,n}}^{N}\ d\nu) \leqslant \lim_{n \rightarrow \infty}(\sum_{i=1}^n \nu(U_i))=\sum_{i<\omega}\nu(U_i).$$
It completes the proof of Claim 2.
\hfill \textit{Claim 2} $\square$

\vspace{1mm}

\vspace{1mm}

Now, since by Claim 2 the set $X$ has full outer measure in $N$, we can use Proposition \ref{Fremlin subspace1} to deduce that
$\int_X f \ d \mu=\int_X R_f^{N}|_X \ d \mu=\int_N R_f^{N} \ d \nu=I(f)$ for each $f \in \mathcal{A}$.
Finally, we consider the obtained measure $\mu$ on $X$ and restrict it to the $\sigma$-algebra $\mathcal{E}$, the smallest $\sigma$-algebra making every function in $\mathcal{A}$ measurable, and denote it by $\rho$, while it is easy to see that on $(X,\mathcal{E},\rho)$, we have $\int_X f \ d \rho=I(f)$ for each $f \in \mathcal{A}$.
It completes the proof of Daniell-Stone theorem. \hfill $\square$

\end{proof}

\subsection{Riesz representation theorem}
In this subsection, we will give a new proof for Riesz Representation theorem.
Note that there are several proofs for Riesz representation theorem via different techniques (such as classical measure theoretic techniques, nonstandard analysis approaches, etc) for example in papers
\cite{GarlingAshortproofRiesz}, \cite{GarlingAnothershortproofRiesz}, \cite{HartigTheRieszrepresentationrevisited}, \cite{RossYetanotherproof} and \cite{ZivaljevicALoebmeasureapproachtoRiesz}.
The proof of Riesz representation theorem we present here is using logic and is similar in many parts to our proof of Daniell-Stone theorem (Theorem \ref{DaniellStone}) we presented above. However, in order for reader to have the proofs of these two theorems independent of each other and also for the sake of completeness and clarity, we present the proof with details although in several parts we refer to the technicalities of the proof of Theorem \ref{DaniellStone}. 

Recall that the Baire $\sigma$-algebra of a topological space
$X$ is the smallest $\sigma$-algebra for which every element of $C(X)$
(the space of continuous functions on $X$) is measurable.
A Baire measure on a topological space is a measure on its Baire $\sigma$-algebra.
We also remind that the well-known Dini's theorem states that if $X$ is a compact topological space, and $(f_n)_{n \in \mathbb{N}}$ is a monotonically decreasing (increasing) sequence of continuous real-valued functions on $X$ converging pointwise to a continuous function $f$, then the convergence is uniform.

\begin{thm} \label{RRT}
\emph{(Riesz representation theorem)}
Let $(X,\tau)$ be a compact Hausdorff topological space and $I$ a positive
linear functional on $C(X)$. Then, there exists a Radon measure $\rho$ on
$X$ such that $I(f)=\int f \ d \rho$ for every $f \in C(X)$.
\end{thm}
\begin{proof}
If $I(1)=0$, then it is not hard to see that $I(f)=0$ for every $f \in C(X)$. This case gives rise to a measure $\mu$ with $\mu(X)=0$.
So, without loss of generality, we may assume that $I(1)=1$.
Let $\mathcal L$ be the language consisting of a constant symbol $c_a$ for each
$a\in X$ and a unary relation symbol $R_f$ for each $f\in C(X)$. 
Also we let universal bound of each relation symbol $R_f$ to be equal to an arbitrary upper bound of the function $|f|$.
Let $T$ be a $\mathcal{L}$-theory consisting of the following expressions (axioms) which can be written as some closed statements in integration logic in the language $\mathcal{L}$ (the reader can get help from Remark \ref{somesampleexpresibleproperties} for stating them).

\begin{enumerate}
\item{$\e(c_a,c_b)=0$ \ \ \ \ \ \ (for every distinct $a,b\in X$),}\label{Rieszaxiomconstantsseparated}
\item{$R_f(c_a)=f(a)$ \ \ \ \ \ \ (for each $a \in X$ and $f \in C(X)$),}\label{Rieszaxiomforconstants}
\item{$R_r(x) \stackrel{a.e}{=} r$ \ \ \ \ \ \ (for each $r \in \mathbb{R}$),}\label{RieszaxiomforconstantR_r_is_r}
\item{$R_{f + g} \stackrel{a.e}{=} R_f + R_g$ \ \ \ \ \ \ (for every $f,g \in C(X)$),}\label{RieszStoneaxiomRf+g=Rf+Rg}
\item{$R_{rf} \stackrel{a.e}{=} rR_f$ \ \ \ \ \ \ (for each $f \in C(X)$ and $r \in \mathbb{R}$),}\label{RieszStoneaxiomRrf=rRf}
\item{$R_{f \vee g} \stackrel{a.e}{=} R_f \vee R_g$ \ \ \ \ \ \ (for every $f,g \in C(X)$),} \label{rieszaxiomforunion}
\item{$\int R_f=I(f)$ \ \ \ \ \ \ (for each $f \in C(X)$),}\label{RieszStoneaxiomintRf=If}
\end{enumerate}

Note that, as similarly explained in the case of Daniell-Stone theorem (Theorem \ref{DaniellStone}), in axioms \ref{RieszStoneaxiomRf+g=Rf+Rg}, \ref{RieszStoneaxiomRrf=rRf} and \ref{rieszaxiomforunion}, the notations $+$, scalar multiplication $rf$ and $\vee$ in the lefts sides of the equalities are referring the corresponding operations in $C(X)$ while in the right sides are addressing the logical connective "+", "." and "max" between formulas.

The proof of the finitely satisfiability of theory $T$ is very similar
to the proof of finitely satisfiability in the Daniell-Stone theorem case we presented before.
The main difference is that in the absence of assumption of order-continuity property for $I$, we use Dini's theorem 
and that (by compactness of $X$) uniform convergence replaces pointwise increasing/decreasing convergence, to conclude that in this case actually 
$\lim$ and $I$ still commute.

Now, in the next step of the proof, we use logical compactness theorem (Theorem \ref{logicalcompactnesstheorem}) to find a model for $T$. Let $(N,\Bc_1,\nu_1)$ be the underlying measure space of that model. Also let $\mathcal{K}$ be the $\mathbb{R}$-vector lattice of functions generated by the family 
$\{R_f^N: f \in C(X)\}$.
Note that every function in $\mathcal{K}$ is the interpretation of a formula and is measurable with respect to the $\sigma$-algebra $\Bc_1$. 
Let $\Bc \subseteq \Bc_1$ be the minimal $\sigma$-algebra on $N$ making every function in $\mathcal{K}$ measurable. Let $\nu:=\nu_1|_{\Bc}$ and consider the measure space $(N,\Bc,\nu)$.
By axioms \ref{Rieszaxiomconstantsseparated} and \ref{Rieszaxiomforconstants} of $T$ and by identifying every $a \in X$ with the interpretation of the constant symbol $c_a$ in $N$, we may assume, without loss, that $X \subseteq N$ and that each $h\in  C(X)$ is the restriction of $R^N_h$ to $X$. Also it is easy to see that restriction of any member of $\mathcal{K}$ to $X$ belongs to $C(X)$.

We aim to show that $X$ has full outer measure in $N$ with respect to the measure $\nu$.
Note that this is very similar to the proof of Claim 2 of the proof of Theorem \ref{DaniellStone}. But for the sake of completeness we give the general idea and mention some slight differences.

Let $\mathcal{C}$ be the Boolean algebra
generated by the sets $\theta^{-1}(r,\infty)$ in $N$ where $\theta \in \mathcal{K}$ and $r \in \mathbb{R}$.
Note that for every $\theta \in \mathcal{K}$ and $r \in \mathbb{R}$, we have $\theta^{-1}(r,\infty)=\theta'^{-1}(0,\infty)$ where $\theta'=\theta-r$, and since $\mathcal{K}$ is a vector lattice, $\theta' \in \mathcal{K}$.
So $\mathcal{C}$ is the Boolean algebra generated by the sets $\theta^{-1}(0,\infty)$ in $N$ where $\theta \in \mathcal{K}$.
By the minimality of $\Bc$ mentioned above, it is easily seen that $\Bc$ is the $\sigma$-algebra generated by $\mathcal{C}$.
So, since $\nu$ is $\sigma$-finite, by a usual extension theorem in measure theory, for example Theorem A p.54 of \cite{HalmosMeasuretheory}, there exists a unique
extension of $\nu |_{\mathcal{C}}$ to $\mathcal{B}$ and it is $\nu$ itself.
But, on the other hand, Carath\'eodory extension theorem (Theorem \ref{caratheodoryextensiontheorem}) extends $\nu |_{\mathcal{C}}$ to $\mathcal{B}$.
It follows that $\nu$ on $\mathcal{B}$ is the same as the measure obtained by the Carath\'eodory extension process from $\nu |_{\mathcal{C}}$.
Hence, by Carath\'eodory extension process explained in the beginning of the paper, for each $U \in \mathcal{B}$ we have
$$\nu(U)=\inf \Big\{\sum_{i < \omega} \nu(U_i): \ \ U \subseteq \bigcup_{i < \omega} U_i, U_i \in \mathcal{C}\Big\}.$$

Let $\mu$ be the induced subspace measure on $X$ by $\nu$.
As mentioned above, we want to show that $\mu(X)=1$ or equivalently, show that $X$ has full subspace measure with respect to $\nu$.
Let $\{U_i\}_{i<\omega} \subseteq \Cc$ be a covering of $X$.
Similar to the argument of Daniell-Stone theorem, it is enough to show that $1 \leqslant \sum_{i < \omega} \nu(U_i)$ and again 
by applying Lemma \ref{changecover}, we may assume that for each $i$, 
$U_i=\theta_i^{-1}(0,\infty)$ for some $\theta_i \in \mathcal{K}$ and moreover, $\nu(\theta_i^{-1}(\{0\}))=0$.
Let $V_i:=U_i \cap X$ for each $i$. Then, clearly $\{V_i\}_{i<\omega}$ is a covering for $X$.
We remind that 
$X$ is being viewed as a subset of $N$ and moreover, as mentioned above, restriction of any member of $\mathcal{K}$ to $X$ belongs to $C(X)$.
So, for each $i <\omega$, we have $V_i=f_i^{-1}(0,\infty)$ where $f_i:=\theta_i|_X \in C(X)$ and moreover, similar to the argument we used in the same part of the proof of Theorem \ref{DaniellStone}, we have $R_{f_i}^{N} \stackrel{a.e}{=} \theta_i$.
Also, the family $\{V_i\}_{i<\omega}$ forms an open covering of $X$.
Thus, by topological compactness, there exists $m < \omega$ such that
$X=V_1 \cup\ldots\cup V_m$.
By Lemma \ref{tendtochar}, for each $i \leqslant m$ there is a particular increasing sequence $(f_{i,n})_{n<\omega}$
of $[0,1]$-valued functions in $C(X)$ converging to $\chi(V_i)$ pointwise 
and that the support of each function in the sequence is a subset of $V_i$.
Now, by a very similar method to the proof of subclaim of the proof of Theorem \ref{DaniellStone}, we can show that
for each $i \leqslant m$ and $n < \omega$, $\int R_{f_{i,n}}^{N}\ d\nu \leqslant \nu(U_i)$.
For every $n < \omega$, let $g_n:=f_{1,n}\vee\ldots\vee f_{m,n}$.
Obviously, every $g_n$ belongs to $C(X)$.
Since $V_1 \cup\ldots\cup V_m$ covers $X$ and for each $i$, the sequence $(f_{i,n})_{n<\omega}$ increases to $\chi(V_i)$,
the sequence $(g_n)_{n<\omega}$ increases pointwise to $1_X$.
So, by using Dini's theorem, $(g_n)_{n<\omega}$ increases uniformly to $1_X$. Now, it can be easily shown that $\lim_{n \rightarrow \infty} I(g_n)=1$.
Again, similar to the proof of Claim 2 of Theorem \ref{DaniellStone}, for each $i \leqslant m$ and $n < \omega$, we have 
$0  \stackrel{a.e}{\leqslant} R_{f_{i,n}}^{N}  \stackrel{a.e}{\leqslant} 1$.
So by using axiom \ref{rieszaxiomforunion}, we have
$R_{g_n}^{N} \stackrel{a.e}{=} \bigvee_{i=1}^m R_{f_{i,n}}^{N}  \stackrel{a.e}{\leqslant} \sum_{i=1}^m R_{f_{i,n}}^{N}$.
Thus, by axiom \ref{RieszStoneaxiomintRf=If} and the mentioned fact that for each $i \leqslant m$ and $n < \omega$, $\int R_{f_{i,n}}^{N}\ d\nu \leqslant \nu(U_i)$, it is followed that for each $n < \omega$,
$$I(g_n)=\int R_{g_n}^{N}\ d\nu \leqslant \sum_{i=1}^m \int R_{f_{i,n}}^{N}\ d\nu \leqslant \sum_{i=1}^m \nu(U_i).$$

Since the above inequalities holds for every $n < \omega$, we have
$1=\lim_{n \rightarrow \infty}I(g_n) \leqslant \sum_{i=1}^m \nu(U_i) \leqslant \sum_{i<\omega}\nu(U_i).$
It follows that $X$ has full subspace measure in $N$ with respect to $\nu$.
Now by Proposition \ref{Fremlin subspace1}, for each $f \in C(X)$ 
$$\int_X f \ d \mu=\int_X R^{N}_f|_X \ d \mu=\int_N R^{N}_f \ d \nu=I(f).$$

By considering the construction of a subspace measure and its $\sigma$-algebra explained in 
Subsection \ref{prelimmeasureintlog},
it is not hard to see  that the $\sigma$-algebra of the subspace measure $\mu$ is exactly the same as the Baire $\sigma$-algebra on $X$.
It follows that the subspace measure $\mu$ is a Baire measure on $X$.
Marik's extension theorem (see \cite{MarikTheBaireBorelmeasure}) states that in a countably paracompact normal topological space, every Baire measure admits a unique regular Borel extension.
Therefore, in particular, $\mu$ on $X$ has a unique regular extension to a Radon measure $\rho$ on $X$. 
Also for each $f \in C(X)$, we have
$\int_X f \ d \rho=\int_X f \ d \mu=I(f)$.
It completes the proof of Riesz representation theorem.
\hfill $\square$

\end{proof}

\vspace{2mm}

{\bf Acknowledgement.}
The author is indebted to Institute for Research in Fundamental Sciences, IPM, for
support. This research was in part supported by a grant from IPM (No.98030116).

\vspace{3mm}

\bibliography{../../references/listbformeasureexistencepaper}  

\bibliographystyle{elsart-num-sort}			

\end{document}